\documentclass{article}
\usepackage{graphicx}
\usepackage{amssymb,latexsym,amsmath,amsthm,amsfonts,graphics}
\usepackage{color}

\newtheorem{thm}{Theorem}
\newtheorem{cor}{Corollary}
\newtheorem{rem}{Remark}
\newtheorem{ex}{Example}
\newtheorem{prop}{Proposition}
\newtheorem{lem}{Lemma}
\newtheorem{df}{Definition}

\def\a{\vec{a}}

\def\x{\tilde{x}}
\def\y{\overline{y}}
\def\C{\mathbb{C}}
\def\R{\mathbb{R}}

\def\Z{\mathbb{Z}}
\def\Q{\mathbb{Q}}
\def\O{\mathbb{O}}
\def\OO{\tilde{\Omega}}
\def\RR{\mathcal{R}}
\def\T{\mathbb{T}}

\def\M{\mathcal{M}}

\def\1{{\rm(i)}}
\def\1{{\rm(i)}}
\def\1{{\rm(i)}}
\def\1{{\rm(i)}}
\def\2{{\rm(ii)}}
\def\3{{\rm(iii)}}
\def\4{{\rm(iv)}}
\def\5{{\rm(v)}}
\def\n{\mathbf{n}}

\title{Dynamical systems defined by polynomials with algebraic properties
\author{
Shigeki Akiyama\thanks{Institute of Mathematics, University of Tsukuba
(akiyama@math.tsukuba.ac.jp)}~,
Xiang Gao\thanks{Faculty of mathematics and statistics,
Hubei University, Wuhan, Hubei, 430062 China (gaojiaou@gmail.com)}~
and
Teturo Kamae\thanks{Advanced Mathematical
Institute, Osaka Metropolitan University, 558-8585 Japan (kamae@apost.plala.or.jp)}}}
\date{}

\begin{document}
\maketitle
\vspace{-1em}
\begin{abstract}
Let $G$ be an additive group and let $St(G)=\prod_{-\infty}^\infty G$ denote the set of infinite two-sided sequences over $G$, which is also called the set of streams over $G$. We are mainly interested in the case that $G$ is the torus $\T=\R/\Z$, identifying each stream $X:=(x_n:n\in\Z)\in St(\T)$ with the formal power series $X(z)=\sum_{n\in\Z}x_nz^n$ in a single variable z.
Assume that $P(z)\in\Z[z]$ is a primitive polynomial,
we investigate the algebraic and dynamical properties of the space of stream zeros
$\Omega_P\subset St(\T)$, defined as the kernel of the action $P(z)$ on $St(\T)$,
where the action $T_P: St(\T)\mapsto St(\T)$ can be viewed as the multiplication in the ring of formal power series over variable $z$ via $P(z)\times X(z)$.

Together with the shift $\sigma$ defined by $(\sigma X)_n=x_{n+1}$ for all $n\in\Z$,
$\Omega_P$ becomes a compact dynamical system $(\Omega_P,\sigma)$ whose symbolic representation is given by the integer part of the image of the action. We proved that there is a one-to-one correspondence between the factorization of polynomial $P(z)$ and a semi-direct sum decomposition of the additive group $\Omega_P$.

Furthermore, we investigate the group of strong automorphisms of $\Omega_P$. We prove that the group of strong automorphisms of $\Omega_P$ can be represented by the matrices from $Saut_P$ and this strong automorphism group is isomorphic to a unit group which determined by the leading coefficients and roots of $P(z)$; by Dirichlet's unit theorem, it is isomorphic to the direct product of cyclic groups of infinite order and a finite cyclic group.
For the quadratic case of $\deg P=2$, we provide an explicit description using the theory of continued fractions. Finally, we study the streams over the Galois field $\text{GF}(q)$, proving that the group of strong automorphisms of $\Omega_P^q$ is a cyclic
group of order $q^k-1$, where $k$ is the degree of $P(z)$.
\end{abstract}

\section{Introduction}
Let $G$ be an additive group, and let $X=(x_n)_{n=-\infty}^{\infty}\in\prod_{-\infty}^\infty G$ be an infinite two-sided sequence such that $x_n\in G$ for all integers $n$.
In this paper, we also write $X=(x_n\in G: n\in\Z)$ and call it a
\textbf{two-sided stream} (or simply, \textbf{stream}) over $G$, we use $St(G)$ to denote the set of all streams over $G$.
Sometimes, a stream $X$ is written as $(\cdots,x_{-2},x_{-1};x_0,x_1,\cdots)$, where the semicolon `;' separates the coordinates for negative indices ($n < 0$) from those for non-negative indices ($n \ge 0$). It is can also be viewed as
as a two-sided formal power series over a variable $z$:
$$X=X(z)=\cdots+x_{-2}z^{-2}+x_{-1}z^{-1}+x_0+x_1z+x_2z^2+\cdots.$$
Note that one-sided infinite sequences over $G$ have been well studied in computer science under the framework of stream calculus, see, for example, \cite{R}.

Let $G$ be a topological commutative ring with  additive identity $0$ and  multiplicative identity $1$, then $St(G)$ forms an additive group under component-wise addition and is also a topological group with respect to the product topology induced by $G$. For convenience, the additive identity of $St(G)$ is denoted by $O=(0^\infty;0^\infty)$,  where $0^\infty$ preceding the semicolon indicates that all terms to the left are zero, and $0^\infty$ following it indicates that all terms to the right are zero.

Motivated by the multiplication of two formal power series, we define the \textbf{convolution product}(or Cauchy product) between two steams $(x_n:n\in\Z)$ and $(y_n:n\in\Z)\in St(G)$ as follows:
\begin{equation}
(x_n:n\in\Z)\times(y_n:n\in\Z)=\left(\sum_{i=-\infty}^\infty x_iy_{n-i}:n\in\Z\right).\label{(1)}
\end{equation}
When the series on the right-hand side converges for all $n\in\Z$, we say that the convolution product exists. Note that this is distinct from the Hadamard product, which will be discussed in Section 5. With respect to this convolution product, the multiplicative identity of $St(G)$ is $I=(0^\infty;1,0^\infty)$.

Given a stream $(x_n:n\in\Z)\in St(G)$, even if $(y_n:n\in\Z)\in St(G)$ such that
$$(x_n:n\in\Z)\times(y_n:n\in\Z)=I$$
exists, it is not necessarily unique. This ambiguity is illustrated by the following example:
$$(0^\infty;1,-1,0^\infty)\times(0^\infty;1^\infty)
=(0^\infty;1,-1,0^\infty)\times((-1)^\infty;0^\infty)=I.$$
Equivalently, in terms of formal power series:
$$(1-t)\times(1+t+t^2+\dots)=(1-t)\times(-\frac{1}{t}-\frac{1}{t^2}-\dots)=1.$$
Because of this example, we should impose appropriate conditions on streams to guarantee both the existence and uniqueness of their convolution product.

We call a stream $(x_n:n\in\Z)\in St(G)$ \textbf{absolutely summable} if the series $\sum_{n\in\Z}x_n$ converges and all rearrangements of its will converge to the same sum.  

\begin{df}{\rm
A polynomial $P(z)=\sum_{n\in S}a_nz^n\in\Z[z]$ is said to be \textbf{primitive} of degree $k$
if there exists a positive integer $k$ with $\{0,k\}\subset S\subset\{0,1,\cdots,k\}$ such that
$a_i\in\Z\setminus\{0\}$ for all $i\in S$ and the greatest common divisor (GCD) of its
coefficients is $\pm1$.

We say that $P(z)$ is \textbf{hyperbolic} if $P(z)$ does not vanish on the unit circle $\{z\in\C:|z|=1\}$, that is, $P(z)\neq0$ for any $z\in\C$ with $|z|=1$.
}\end{df}

Let $P(z)=\sum_{n\in S}a_nz^{n}\in\Z[z]$ be a primitive polynomial as in Definition 1 and let $\T:=\R/\Z$ be the torus.
Then, $P(z)$ acts continuously on $St(\T)$ (denoted by $T_P$) via the multiplication of formal power series, which corresponds to
the convolution product of streams defined in (\ref{(1)}). Specifically,
\begin{multline*}
T_P:(x_n:n\in\Z)\mapsto(a_n:n\in S)\times(x_n:n\in\Z)\\
=(a_kx_{n-k}+\cdots+a_1x_{n-1}+a_0x_n:n\in\Z)\in St(\T).
\end{multline*}
We study the kernel of this action, denoted by $\Omega_P$, which we call the \textbf{stream zeros} of $P(z)$:
$$\Omega_P:=\left\{(x_n:n\in\Z)\in St(\T):P(z)\times(x_n:n\in\Z)=(0^\infty;0^\infty)\right\}.$$
In other words, $\Omega_P$ is the set of streams over $\T$ that are annihilated by polynomial $P(z)$, meaning their convolution product with $P(z)$  yields the zero stream $O=(0^\infty;0^\infty)$. Equivalently, $\Omega_P$ consists of all sequences $(x_n)_{n\in\Z}$ on $\T$ that satisfy the linear recurrence relation $\sum_{j=0}^ka_jx_{n-j}=0$ for all $n\in\Z$.

The remainder of this paper is organized as follows.
In Section 2, we establish some basic properties
of the convolution product which are useful for subsequent sections. In Section 3,
we analyze $\Omega_P$ from a dynamical system viewpoint. In Section 4, we introduce
the notion of lift of a steam zeros on $\T$ alongside the lifting matrix of a polynomial $P(z)$, we prove Theorem 3 which provide a construction method for a lift whenever Condition (\#) (Definition 3 in Section 4) is satisfied.  Moreover, we prove that any primitive $P(z)$ satisfies Condition (\#), a result of independent interest that generalizes earlier work by M. Newman \cite{N} and \cite{Z} for lifting matrices of size one.

In Section 5, we utilize polynomial resultant, generalized Vandermonde's determinant, and Hadamard product to establish one-to-one correspondence between the factorization of $P(z)$ and the semi-direct sum decomposition of the additive group $\Omega_P$. In Section 6, the strong automorphism group of $\Omega_P$ is introduced and studied. We prove that the group of strong automorphisms of $\Omega_P$ can be represented by the matrices from $Saut_P$ (Definition 4). Using Dirichlet's unit theorem, we prove Theorem 7, which shows that there are closely relationship between the algebraic structure of $Saut_P$ and the group of units formed from the polynomial root of $P(z)$, this have also been investigated in \cite{BLR,KKS}. We provide many examples to illustrate these results, while alternative perspectives on the automorphism groups of general symbolic spaces (such as symbolic complexity) can be found in \cite{CK, H}.

In Section 7, we discuss the quadratic polynomial case ($\deg P=2$), showing that $Saut_P$ is determined by the discriminant of $P(z)$. For the degree larger than 2, we obtain partial results that indicate a close connection to the theory of
Diophantine equations, generalizing the classical Pell's equation seen in the quadratic case. We provide examples to illustrate all theses theorem.
In the last section, we focus on stream zeros over the Galois field (or finite field) $\text{GF}(q)$ of order $q$, where $q$ is a prime power, under the action of a polynomial $P(z)$. We prove that if $P(z)$ is an irreducible polynomial of degree $k$, then $Saut_P$ is a cyclic group of order $q^k-1$.

\section{Convolution product and its inverse}
In this section, we present some basic properties of the
convolution product of $St(\C)$. As one can see from the previous discussion, the convolution inverse may not be unique, we thus wish to find some conditions that can guarantee the uniqueness of the convolution inverse. The following statement \4 of Theorem 1 is known (\cite{AKK}) by complex analytic methods, here we give another
elementary proof.
\begin{thm}
Let $U=(u_n:n\in\Z),~V=(v_n:n\in\Z),~W=(w_n:n\in\Z)$ be elements in $St(\C)$.\\
\1 $U\times V=V\times U$ whenever either products $U\times V$ or $V\times U$ exists.
If among these three streams $U,V$ and $W$, two are absolutely summable and the remaining one is bounded, then the iterated convolution products exist and satisfy:
$$(U\times V)\times W=U\times(V\times W).$$
\2~If $U$ is absolutely summable, an absolutely summable $V$ satisfying
$U\times V=I$ or $V\times U=I$ (if exists) is unique. In this case, we have
$U\times V=V\times U=I$ and $V$ is called the convolution inverse of $U$, denoted by $U^{-1}$.\\
\3~If both $U$ and $V$ are absolutely summable, then so is $U\times V$.
Furthermore, if both $U^{-1}$ and $V^{-1}$ exist, then
$(U\times V)^{-1}$ exists with $(U\times V)^{-1}=V^{-1}\times U^{-1},$
 and $(U^{-1})^{-1}=U$. \\
\4~Any primitive polynomial $P(z)=(a_n:n\in\Z)$
has a unique absolutely summable inverse $P(z)^{-1}$ in $St(\R)$ provided that $P(z)$ is hyperbolic.
\end{thm}

\begin{proof} (i) The first statement is clear. For the second statement,
let $U,V,W$ satisfy the required condition.
Then for any $n\in\Z$, we prove that
$$
\sum_{i,j,k\in\Z\atop{i+j+k=n}}|u_iv_jw_k|<\infty.
$$
By symmetry, we may assume that $U,V$ are absolutely summable and $W$
is bounded by $L$, say. Then, we have
$$
\sum_{i,j,k\in\Z\atop{i+j+k=n}}|u_iv_jw_k|<L\sum_{i,j\in\Z}|u_iv_j|
=L\sum_{i\in\Z}|u_i|\sum_{j\in\Z}|v_j|<\infty.
$$
Thus, $\sum_{i,j,k\in\Z\atop{i+j+k=n}}u_iv_jw_k$ converges absolutely.
Hence, we have
$$\sum_{k\in\Z}\left(\sum_{i,j\in\Z\atop{i+j=n-k}}u_iv_j\right)w_k
=\sum_{i\in\Z}u_i\left(\sum_{j,k\in\Z\atop{j+k=n-i}}v_jw_k\right)
=\sum_{i,j,k\in\Z\atop{i+j+k=n}}u_iv_jw_k
$$
and $(U\times V)\times W=U\times(V\times W)$.

(ii) We prove the uniqueness since the other statements follow from the
first statement in \1. Indeed, if $U\times V_0=U\times V_1=I$, then by \1, we have
$$V_1=I\times V_1=(U\times V_0)\times V_1=(V_0\times U)\times V_1=V_0\times (U\times V_1)=V_0.$$

(iii) It is clear.

(iv) Let $P(z)=\sum_{n\in S}a_nz^n$ with a finite nonempty set $S=\{n\in\Z:a_n\ne0\}$. As we know, every polynomial with real coefficients can be factored into a product of linear factors and/or irreducible quadratic factors (of degree two) with real coefficients. These further factor into two linear factors with complex coefficients of the form $z-\omega$ and $z-\overline{\omega}$, here $\overline{\omega}$ is the complex conjugate of $\omega$.

By factoring polynomial $P(z)$ as above and applying (ii), it suffices to show that
$P(z)$ of the elementary forms, either non-zero constant $C$, $z$ or $z-\omega$ with $\omega\in\C$ and
$0<|\omega|\ne1$ (by assumption that $P(z)$ is hyperbolic), possess an absolutely summable inverse. \\
Case 1: $P(z)=C\ne0$. Since $C=(0^\infty;C,0^\infty)$,
$C^{-1}=(0^\infty;C^{-1},0^\infty)$ becomes the absolutely summable inverse. \\
Case 2: $P(z)=z$. One can easily verifies that $z$ and $z^{-1}$ are
absolutely summable and inverses of each other. This can be seen from the fact
$(0^\infty1;0^\infty)=(0^\infty;0,1,0^\infty)^{-1}$.\\
Case 3: $P(z)=z-w$, where $w\in\R$ is a real number, if $|\omega|>1$,
it is obvious that
$$
(z-\omega)^{-1}=(-1)\omega^{-1}(1-\omega^{-1}z)^{-1}
=-\omega^{-1}-\omega^{-2}z-\omega^{-3}z^2-\cdots,
$$
Observe that $\sum_n\omega^{-n}$ is absolutely summable, so $(0^\infty;-\omega^{-1},-\omega^{-2},-\omega^{-3},\cdots)$ is the
absolutely summable inverse of $z-\omega$;
When $P(z)=z-w$ with $|\omega|<1$, similarly, we have
$$
(z-\omega)^{-1}=z^{-1}(1-\omega z^{-1})^{-1}=z^{-1}+\omega z^{-2}
+\omega^2 z^{-3}+\cdots,
$$
now $(\cdots,\omega^2,\omega,1;0^\infty)$ becomes the absolutely summable inverse of $z-\omega$.

Case 4: $P(z)=z-w$, where $\omega\in\C\setminus \R$,  we also consider the complex conjugate $\overline{\omega}$ of $\omega$, since $|\overline{\omega}|=|\omega|$, there is an analogous statement for $\overline{\omega}$ as Case 3, by (iii), $((z-\omega)(z-\overline{\omega}))^{-1}=(z-\omega)^{-1}(z-\overline{\omega})^{-1}\in St(\R)$. Thus, $P(z)$ has the absolutely summable inverse $P(z)^{-1}\in St(\R)$.
\end{proof}

\section{Stream zeros and dynamical systems}

Let $P(z)$ be a primitive polynomial over $\Z$. Clearly, the set of stream zeros $\Omega_P$ is a compact
$\Z$-subgroup of $St(\T)$. Moreover, the pair $(\Omega_P,\sigma)$ becomes a compact dynamical system, where $\sigma$ is the shift defined by $(\sigma X)_n=x_{n+1}$ for all $n\in\Z$, we called it the dynamical system induced by $P(z)$.

Given a primitive polynomial $P(z)=\sum_{n\in S}a_nz^n$ with a finite support $S\subset\Z$ as in Definition 1. Define the lower and upper coefficient sums
\begin{equation}\label{(3)}
k_*:=\sum_{n\in S,a_n<0}a_n\mbox{ and }k^*:=\sum_{n\in S,a_n>0}a_n.
\end{equation}
Put
$$K_P:=\{k_*+1,k_*+2,\cdots,k^*-1\}\subset \Z.$$
For $x\in\R$ (or $\R/\Z$), let $\{x\}\in[0,1)$ denote its fractional
part or the representative of $x$ in $[0,1)$.
While $[0,1)$ is not an additive group, we can consider the streams space $St([0,1))$
analogously as an additive group.

Define a mapping
$T_P:\Omega_P\to K_P^\Z$ by
$$T_P((x_n:n\in\Z)):=P(z)\times(\{x_n\}:n\in\Z)=
\left(\sum_{i\in S}a_i\{x_{n-i}\}\in K_P:n\in\Z\right).$$
Note that the resulting streams $P(z)\times(\{x_n\}:n\in\Z)\in St(\Z)$, because
$P(z)\times(x_n:n\in\Z)=O$ as an element in $St(\T)$.

Let $\Xi_P=T_P(\Omega_P)$.
Clearly, $\Xi_P$ is a
shift-invariant compact subset of $K_P^\Z$, so that
$(\Xi_P,\sigma)$ forms a symbolic dynamical system.
Assuming the formal inverse $P^{-1}$ exists in $St(R)$.
Define the backward mapping $T_{P^{-1}}:\Xi_P\to St(\T)$ so that
$$
T_{P^{-1}}((\xi_n:n\in\Z)):=\pi \left(P^{-1}(z)\times(\xi_n:n\in\Z)\right),
$$
where
$\pi:[0,1)^\Z\to\T^\Z$ is the coordinate-wise projection.
The following theorem is essentially known. We sketch a proof here, for more details, see \cite{AKK}.

\begin{thm}
Let $P(z)=\sum_{n\in S}a_nz^n$ be a primitive polynomial of degree $k$.
Assume that $P(z)$ is hyperbolic. Then, the dynamical systems
$(\Omega_P,\sigma)$ and $(\Xi_P,\sigma)$ are conjugate via the
map $T_P$, which is a Borel isomorphism.

Moreover, a necessary and sufficient condition for a stream
$\xi=(\xi_n:n\in\Z)\in St(\Z)$ belongs to $\Xi_P$ is
that $(P(z)^{-1}\times\xi)_n\in[0,1)$ for all $n\in\Z$.

\begin{figure}[h]
\setlength{\unitlength}{0.4mm}
\begin{picture}(300,60)(-100,25)
\put(24,85){\vector(1,0){37}}
\put(2,73){\vector(0,-1){37}}
\put(-2,36){\vector(0,1){37}}
\put(85,73){\vector(0,-1){37}}
\put(89,36){\vector(0,1){37}}

\put(24,26){\vector(1,0){37}}

\put(-8,80){$\Omega_P$}
\put(78,80){$\Omega_P$}
\put(-4,22){$\Xi_P$}
\put(83,22){$\Xi_P$}

\put(37,75){$\sigma$}
\put(40,29){$\sigma$}
\put(-29,53){$T_{P^{-1}}$}
\put(92,53){$T_{P^{-1}}$}
\put(8,53){$T_P$}
\put(65,53){$T_P$}
\end{picture}
\caption{The commutative diagram between $(\Omega_P,\sigma)$ and $(\Xi_P,\sigma)$.}
\end{figure}
\end{thm}
\begin{proof}
Since $P(z)$ is a primitive polynomial, from (iv) of Theorem 1, this implies that the existence of $T_{P^{-1}}$.  By direct computation, we can show $(\Omega_P,\sigma)$ and $(\Xi_P,\sigma)$ are conjugate.

Now we prove the second half of the statement. To show its necessity,
assume that $\xi\in\Xi_P$, by the definition of $\Xi_P$, then there exists $(x_n:n\in\Z)\in\Omega_P$ such that
$\xi=P(z)\times(\{x_n\}:n\in\Z)$. Therefore, we have
$$
P(z)^{-1}\times\xi=P(z)^{-1}\times P(z)\times(\{x_n\}:n\in\Z)
=(\{x_n\}:n\in\Z)\in St(\Z).
$$

For sufficiency, assume that $Y:=P(z)^{-1}\times\xi\in[0,1)^\Z$, choose a stream $X=((x_n): n\in\Z)\in St(\T)$ such that
$Y=\{X\}:=(\{x_n\}:n\in\Z)$,
since $P(z)\times Y=\xi\in K_P^\Z$, and $K_P\subset\Z$, so we have
$P(z)\times X=(0^\infty;0^\infty)\in St(\T)$, this means $X\in\Omega_P$.
Clearly, $\xi=P(z)\times Y=P(z)\times\{X\}$,
this implies that $\xi\in\Xi_P$.
\end{proof}

\begin{rem}{\rm
We call $(\Xi_P,\sigma)$ the \textbf{symbolic representation} of
$(\Omega_P,\sigma)$.
To recover the dynamical system on $St(\T)$ from its symbolic representation
is known as Kaneko's inverse formula~\cite{AKK}.
}\end{rem}

\begin{ex}{\rm
Let $P(z)=z^2-3z+1$. Then, we have
$$
\Omega_P=\{(x_n:n\in\Z)\in St(\T):x_{n+2}=3x_{n+1}-x_n
~\mbox{(mod $\Z$) for any }n\in\Z\}.
$$
We can rewrite this relation in matrix form as:
\begin{align*}
\Omega_P=&\left\{(x_n:n\in\Z)\in St(\T):
\left(\begin{array}{c}x_{n+1}\\x_{n+2}\end{array}\right)\right.\\
&\left.=\left(\begin{array}{cc}0&1\\-1&3\end{array}\right)
\left(\begin{array}{c}x_n\\x_{n+1}\end{array}\right)
~\mbox{(mod $\Z$) for any }n\in\Z\right\}.
\end{align*}
Observe that $P(z)$ can be factored as
$$
P(z)=z^2-3z+1=(z-\omega_+)(z-\omega_-),\quad \mbox{here}~\omega_\pm:=\frac{3\pm\sqrt{5}}{2}.
$$
Since $|\omega_+|=|\frac{3+\sqrt{5}}{2}|>1$, we expand $(z-\omega_+)^{-1}$ as:
$$
(z-\omega_+)^{-1}=(-1)\omega_+^{-1}(1-\omega_+^{-1}z)^{-1}
=-\omega_+^{-1}-\omega_+^{-2}z-\omega_+^{-3}z^2-\cdots,
$$
Similarly, since $|\omega_-|=|\frac{3-\sqrt{5}}{2}|<1$, we expand $(z-\omega_-)^{-1}$ as:
$$
(z-\omega_-)^{-1}=z^{-1}(1-\omega_- z^{-1})^{-1}=z^{-1}+\omega_- z^{-2}
+\omega_-^2 z^{-3}+\cdots.
$$
Using the geometric series sum,
$$\omega_+^{-1}\sum_{n=0}^\infty(\omega_+^{-1}\omega_-)^n
=\omega_+^{-1}\frac{1}{1-\omega_+^{-1}\omega_-}=\frac{1}{\omega_+-\omega_-}
=\frac{1}{\sqrt{5}},$$
we arrive at the following Kaneko's inverse formula:
\begin{align*}
&P(z)^{-1}=(z-\omega_+)^{-1}\times(z-\omega_-)^{-1}\\
&=(0^\infty;-\omega_+^{-1},-\omega_+^{-2},-\omega_+^{-3},\cdots)\times
(\cdots,\omega_-^3,\omega_-^2,\omega_-,1;0^\infty)\\
&=-\frac{1}{\sqrt{5}}(\cdots,\omega_-^3,\omega_-^2,\omega_-,1;\omega_+^{-1},
\omega_+^{-2},\omega_+^{-3},\cdots).
\end{align*}
It is interesting that this formula can also be derived via partial fraction decomposition:

\begin{align*}
&P(z)^{-1}=(z-\omega_+)^{-1}\times(z-\omega_-)^{-1}=\frac{1}{\omega_+-\omega_-}\{\frac{1}{z-\omega_+}-\frac{1}{z-\omega_-}\}\\
&=-\frac{1}{\sqrt{5}}\left\{\frac{1}{\omega_+}\frac{1}{1-z\omega_+^{-1}}+\frac{1}{z}\frac{1}{1-z^{-1}\omega_-}\right\}\\
&=-\frac{1}{\sqrt{5}}\left\{\frac{1}{\omega_+}\sum_{n=0}^\infty(\omega_+^{-1}z)^n+\frac{1}{z}\sum_{n=0}^\infty(\omega_-z^{-1})^n\right\}\\
&=-\frac{1}{\sqrt{5}}(\cdots,\omega_-^3,\omega_-^2,\omega_-,1;\omega_+^{-1},
\omega_+^{-2},\omega_+^{-3},\cdots).
\end{align*}

\begin{figure}[h]
\setlength{\unitlength}{0.25mm}
\begin{picture}(150,165)(-150,-10)
\put(0,0){\vector(1,0){170}}
\put(0,0){\vector(0,1){170}}
\put(0,150){\line(1,0){150}}
\put(150,0){\line(0,1){150}}

\put(-35,-13){$(0,0)$}
\put(-36,46){$(0,\frac13)$}
\put(-36,96){$(0,\frac23)$}

\put(176,-3){$x_{n}$}
\put(134,-15){$(1,0)$}
\put(156,45){$(1,\frac13)$}
\put(156,95){$(1,\frac23)$}
\put(156,144){$(1,1)$}

\put(-15,174){$x_{n+1}$}
\put(-35,144){$(0,1)$}

\put(0,0){\line(3,1){150}}
\put(0,50){\line(3,1){150}}
\put(0,100){\line(3,1){150}}

\put(90,10){$1$}
\put(75,45){0}
\put(60,90){$-1$}
\put(40,130){$-2$}
\end{picture}
\caption{Symbolic representation of $St([0,1))$: $\xi_{n+2}\in \{-2,-1,0,1\}.$}
\end{figure}
Identifying $\T$ with $[0,1)$, $\Omega_P$ can be viewed as the set of trajectories of the dynamical system on $[0,1)\times[0,1)$ with the transformation
$f:(x,y)\mapsto(y,3y-x)~(\rm{mod}~1)$. That is, $x_{n+2}=\{3x_{n+1}-x_n\}$ for all $n\in\Z$ (remember that $\{~\}$ is the fractional part).
Let $K_P=\{-2,-1,0,1\}$, and define $\xi_{n+2}=x_n-3x_{n+1}+x_{n+2}$, then the symbolic representation of $(x_n\in[0,1):n\in\Z)$
is is given by the sequence $(\xi_n)_{n=-\infty}^{\infty}$, where $\xi_n\in K_P$, for each $n\in\Z$.
Let $[x]$ denote the integer part of $x$, i.e., $[x]=x-\{x\}$, where $\{x\}$ represents the fractional part of $x$. It follows that
$$\xi_{n+2}=x_n-3x_{n+1}+x_{n+2}=-[3x_{n+1}-x_n]\in K_P.$$ as illustrated in Figure 2.

The topological entropy of this dynamical systems is given by
$$\lim_{n\to\infty}\frac{1}{n}\log\#\Xi_P|_{\{1,2,\cdots,n\}},$$
where $\#\Xi_P|_{\{1,2,\cdots,n\}}$ is the number of admissible (allowed) sequences of length $n$ in $\Xi_P$.
As is well known, this entropy coincides with the sum of logarithms of the eigenvalues greater than 1 of the
companion matrix, which is also equal to the logarithms of Mahler measure of polynomial $P(z)$, see \cite{C} or examples 12.1 of \cite{KK}.  In our case, the companion matrix
$\left(\begin{array}{cc}0&1\\-1&3\end{array}\right)$ has eigenvalues
$\frac{3\pm\sqrt{5}}{2}$, and hence the topological entropy is  $\log\frac{3+\sqrt{5}}{2}$.  For further dynamical properties, we refer to D. Lind and B. Marcus\cite{LM}, K.Schmidt \cite{S}.
}\end{ex}

\section{Lifts of stream zeros}
For any integers $n \leq m$, let $[n,m]$ denotes the set of all integers $j$ with $n\leq j\leq m$, which is the same thing as the set $\{n,n+1,\dots,m-1,m\}$.
Let $P(z)=a_kz^k+a_{k-1}z^{k-1}+\cdots+a_0\in\Z[z]$ be a primitive polynomial of degree $k$ and $N$ be a positive integer (denote by $\Z_{>0}$). Define
$$
\OO_P:=\{(\x_n\in\R:n\in\Z):a_k\x_{n-k}+a_{k-1}\x_{n-k+1}+\cdots+a_0\x_n=0
~(\forall n\in\Z)\}
$$
and its restriction to finite interval $[-N,N]$ as:
$$
\OO_P|_N:=\{(\x_n\in\R:n\in[-N,N]):(\x_n\in\R:n\in\Z)\in\OO_P\}.
$$

\begin{df}{\rm
A stream $(x_n:n\in\Z)\in St(\T)$ is said to have a \textbf{lift} in $\OO_P$
if for any positive integer $N\in\Z_{>0}$, there exists a \textbf{local lift} $(\x_n\in\R:-N\le n\le N)\in\OO_P|_{N}$ such that
$\pi(\x_n)=x_n$ for any $n\in[-N,N]$, where $\pi:\R\to\T$ is the canonical projection.
We say that $\Omega_P$ has a lift from $\T$ to $\R$ if every stream in $\Omega_P$ has a lift in $\OO_P$.

}\end{df}
One may ask under what algebraic conditions on polynomial $P(z)$ such that the stream zeros set $\Omega_P$ possesses a lift. We answer it in the remainder of this section.

For any positive integer $l$, define the $l\times(k+l)$-matrix $M_P^{(l)}$ associated with $P(z)$:
\begin{equation}
M_P^{(l)} = \begin{pmatrix} a_k & a_{k-1} & \cdots & a_0 & 0 & \cdots & \cdots & \cdots & 0 \\ 0 & a_k & a_{k-1} & \cdots & a_0 & 0 & \cdots & \cdots & 0 \\ \vdots & \vdots & \vdots & \vdots & \vdots & \vdots & \vdots & \vdots & \vdots \\ 0 & \cdots & \cdots & 0 & a_k & a_{k-1} & \cdots & a_0 & 0 \end{pmatrix}
\label{(4)}
\end{equation}
\vspace{-1em}

\begin{df}{\rm
We say that $P(z)$ satisfies \textbf{Condition (\#)} if for any positive integer $l$, there exists an
integral matrix $R$ (not necessarily unique) of size $k\times(k+l)$ such that the determinant of the $(k+l)\times(k+l)$ square matrix $A^{l}$ equals $\pm1$, where the matrix $A^{l}$ is formed by augmenting $M_P^{(l)}$ with $R$:

$$A_{l}=\left(\begin{array}{c}M_P^{(l)}\\R\end{array}\right).$$
}\end{df}
Partition the inverse matrix $A_l^{-1}$ into two block submatrices: $$A_l^{-1}=(\tilde{L}\mid L)$$
where $\tilde{L},L$ are
$(k+l)\times l$ sub-matrix and $(k+l)\times k$ sub-matrix of matrix $A_{l}^{-1}$, respectively.
The $(k+l)\times (k+l)$ square matrix  $\Lambda:=LR$, which is exactly the matrix multiplication of two matrices $L$ and $R$, is then called an \textbf{lifting matrix} of $P(z)$ of size $l$.

For instance, let $P(z)=3z-2$, so the degree $k=1$. For $l=1$, the matrix $M_P^{(1)}=(3~~-2)$. Then, we can choose the $1\times2$ matrix $R=(-1~~1)$ such that $\det A_{l}=\det\left(\begin{array}{c}M_P^{(1)}\\R\end{array}\right)=\det\begin{pmatrix}
  3& -2 \\
  -1& 1
\end{pmatrix}=1.$
The inverse matrix of $A_{l}$ is exactly
$\begin{pmatrix}
  1& 2 \\
  1& 3
\end{pmatrix}$.
So the matrix $L=\begin{pmatrix}
  2 \\
  3
\end{pmatrix}.$
Then the square $2\times 2$ matrix $$\Lambda=LR=
\begin{pmatrix}
  -2& 2 \\
  -3& 3
\end{pmatrix}
$$
becomes a lifting matrix of $P(z)$ of size 2. In fact,
let $(1/5,4/5)$ be a segment of a stream in $\Omega_P$ satisfying
$3\times(1/5)-2\times(4/5)=-1\in\Z$. Then, we have
$$
\begin{pmatrix}
  -2& 2 \\
  -3& 3
\end{pmatrix}
\begin{pmatrix}
  1/5 \\
  4/5
\end{pmatrix}
=
\begin{pmatrix}
  6/5 \\
  9/5
\end{pmatrix}
\equiv
\begin{pmatrix}
  1/5 \\
  4/5
\end{pmatrix}
~~(\mbox{mod}~\Z),
$$
and $3\times(6/5)-2\times(9/5)=0$.

Roughly speaking, Condition ($\#$) requires that one can find at least one such matrix $R$ that makes the resulting square matrix $A_l$ unimodular (its determinant is $\pm 1$), ensuring that $A_l$ has an inverse with integer entries.
For the lifting matrix is of size lager than 2, see the example 2 in this section. However, when the degree of polynomial $P(z)$ is larger than 2, the existence of lifting matrix is not easy. Before to show a general statement, first we present a useful lemma due to S. Akiyama \cite{A}.

\begin{lem}{\rm \cite{A}}
In the hyperplane $H$ of $\R^{k+l}$ spanned by the row vectors of the matrix $M_P^{(l)}$,
any integer point in $H$ can be written as a linear combination of the row vectors
of $M_P^{(l)}$ with integer coefficients.
Consequently, the parallelepiped spanned by the $l$ row vectors of $M_P^{(l)}$ contains
integer points only at its vertices.
\end{lem}
Return to the example of $P(z)=3z-2$, the lemma is noting but to say that $\Z=\{3m+2n: m,n\in\Z\}$( using $GCD(-2,3)=1$), which is exactly the statement of B\'{e}zout's Lemma.

\begin{proof}
It is enough to prove that if $\eta=(x_1,x_2,\cdots,x_l)$ is a row vector of size $l$ with real entries such that $\eta M_P^{(l)}$ is an integer vector, then $\eta$ itself must be an integer vector. If this holds, it implies that any integer point $(y_1,y_2,\cdots,y_{k+l})$ in the space spanned by the row vectors of $M_P^{(l)}$ can be expressed as an integer linear combination of those rows.

First, we prove that $\eta$ has only rational entries.
Suppose to the contrary that $\eta$ contains irrational elements. Among these elements we choose $x_i$ such that the index $i\in\{1,2,\dots,l\}$ is the smallest one.  Now we consider the $i$-th
element of the vector $\eta M_P^{(l)}$, which is irrational, this contradicts the assumption that $\eta M_P^{(l)}$ is an integer vector. Thus, $\eta$ must be a rational vector.

In order to show $\eta$ is an integer vector, it suffices to prove that for any prime number $p$,
$\nu_p(x_i)\ge 0$ for all $i=1,\cdots,l$, where $\nu_p$ is the discrete valuation at
$p$ on rational field $\Q$. Suppose to the contrary that $\nu_p(x_i)<0$ for some prime $p$ and some $i\in\{1,2,\dots,l\}$.
Let $\nu_p(x_i)$ attains the minimum at $x_i$, let $m_0=\min_i \nu_p(x_i)<0$. Since there are may many possibilities $i$ such that $\nu_p(x_i)$ attains the minimum, we consider the smallest one, say $i_0$. That is,  $i_0\in\{1,\cdots,l\}$ be the (unique) one
such that
$$
\nu_p(x_{i_0})=m_0<\nu_p(x_i)~\mbox{for all}~1\le i<i_0~\mbox{and}~
\nu_p(x_{i_0})\le\nu_p(x_i)~\mbox{for all}~i_0<i\le l.
$$

Since $GCD(a_k,\cdots,a_1,a_0)=1$, there exists some
$j\in\{0,1,\dots,k\}$ such that $\nu_p(a_j)=0$. Similarly, consider the smallest index $j_0$ such that
$$
\nu_p(a_{j_0})=0<\nu_p(a_j)~\mbox{for all}~0\le j<j_0.
$$
Notice that $k+i_0-j_0\in[1,k+l]$, and the $(k+i_0-j_0)$-th column of the matrix $M_P^{(l)}$ is $(a_{j_0-i_0+1}, \dots, a_{j_0-i_0+l})^T$, with the convention that $~
a_j=0~\mbox{if}~j\notin\{0,1,\cdots,k\}$, and we define $\nu_p(0) = \infty$. So the $(k+i_0-j_0)$-th element of $\eta M_P^{(l)}$, denoted by $\xi$, becomes
$$\xi=x_1a_{j_0-i_0+1}+\cdots+x_{i_0}a_{j_0}+\cdots+x_la_{j_0-i_0+l}.$$
From the definition of $i_0$ and $j_0$, and using the fact that for any integer $n\neq 0$, $\nu_p(n)$ is always a non-negative integer,
it follows that $$\nu_p(x_ia_{j_0-i_0-i})>\nu_p(x_{i_0}a_{j_0})~\mbox{for all}~i\ne i_0.$$
 Using the properties that for $x,y\in\Q$,  $\nu_p(xy)=\nu_p(x)+\nu_p(y),$ and $\nu_p(x+y)\geq \min\{\nu_p(x),\nu_p(y)\}$, the inequality becomes an equality when $\nu_p(x) \neq \nu_p(y)$.
So we have
$$\nu_p(\xi)=\nu_p(x_{i_0})+\nu_p(a_{j_0})=m_0+0=m_0<0.$$ Thus, $\xi$ is not an integer, which contradicts the assumption again, so the theorem follows.
\end{proof}

\begin{thm}
Every primitive polynomial $P(z)$ satisfies {\rm Condition (\#)}, and consequently, the stream zero set $\Omega_P$ always admits a lift.
\end{thm}
More precisely, for any large $N\in\Z_{>0}$ and $l=2N-k+1$,
let $\Lambda$ be a lifting matrix of size $l$ constructed as in Definition $3$.
Then, for any stream $(x_n:n\in\Z)\in\Omega_P$, the vector defined by
$\Lambda\mathbf{x}$
constitutes a local lift of $(x_n:-N\le n\le N)$ in $\tilde{\Omega}_P$, where $\mathbf{x}$ is the column vector of the truncated sequence $(x_n)_{-N \le n \le N}$ (over $[-N, N]$), i.e, $\mathbf{x}:=(x_{-N}, x_{-N+1}, \dots, x_N)^T.$

\begin{proof}
Let $l\in\{1,2,\cdots\}$ be arbitrary.
To prove Condition (\#), it is sufficient to prove that there exists a
$l\times(k+l)$-integer matrix $R$ such that the set of row vectors
of $\left(\begin{array}{c}M_P^{(l)}\\R\end{array}\right)$ is linear independent and the
parallelepiped spanned by them contains the integer points only at vertices,
since then the volume of the parallelepiped is 1 and hence,
$\det\left(\begin{array}{c}M_P^{(l)}\\R\end{array}\right)=\pm1$.

We construct $R$ inductively starting from the first row. By
Lemma 1, $M_P^{(l)}$ has the property that the rows are linearly independent and
the parallelepiped spanned by them, say $\Gamma_0$, contains the integer points
only at vertices.
Then, we add any integer vector, say $\gamma$, independent of the row vectors of $M_P^{(l)}$.
Let $\Gamma$ be the parallelepiped spanned by $\Gamma_0$ and $\gamma$.
Since both $\Gamma_0$ and $\gamma+\Gamma_0$ contain the integer points only at vertices,
all the integer points in $\Gamma$ which are not the vertices are properly in between
$\Gamma_0$ and $\gamma+\Gamma_0$.

Let $\gamma'$ be any one of them and $\Gamma'$ be the
parallelepiped spanned by $\Gamma_0$ and $\gamma'$. Then the volume of $\Gamma'$
and the number of integer points which are not the vertices is strictly smaller than those
of $\Gamma$. By this process, we finally find an integer vector
$\xi_1$ independent of the row vectors of $M_P^{(l)}$ such that the parallelepiped $\Gamma_1$ spanned by $\Gamma_0$
and $\xi_1$ has the property that all the integer points in it
are only those at the vertices.

Repeating this process described above, we can find a
$(k+l)$-dimensional parallelepiped $\Gamma_k$ in $\R^{k+l}$ which is spanned by
$\Gamma_0$ and $\xi_1,\xi_2,\cdots,\xi_k$ such that
there are no integer points in $\Gamma_k$ other than the vertices.
Thus, the volume of $\Gamma_k$ is $1$ and
$\det\left(\begin{array}{c}M_P^{(l)}\\R\end{array}\right)=\pm1$,
where $R=(\xi_1, \xi_2, \dots, \xi_k)^T.$

Now we prove the existence of the lift of streams zero $\Omega_P$. Let $N\in\Z_{>0}$ be large enough. Put $l=2N-k+1$.
Our aim is to show that for any $(x_n:n\in\Z)\in\Omega_P$,
$\Lambda(x_{-N}, x_{-N+1}, \dots, x_N)^T$ is a local lift of $(x_n:-N\le n\le N)$ in $\tilde{\Omega}_P$.
Without loss of generality, assume that $x_n\in[0,1)~\mbox{for all}~n\in\Z$.
Let
$$
\xi_n=a_kx_{n-k}+a_{k-1}x_{n-k+1}+\cdots+a_0x_n\in\Z~~\mbox{for all}~n\in\Z.
$$
Let $A_l=\left(\begin{array}{c}M_P^{(l)}\\R\end{array}\right)$ be as in Definition 3.
Then, we have
\begin{equation}
A_l\left(\begin{array}{c}x_{-N}\\x_{-N+1}\\\vdots\\x_{-N+l-1}\\x_{-N+l}\\\vdots\\
x_{-N+l+k-1}\end{array}\right)
=\left(\begin{array}{c}\xi_{-N+k}\\\xi_{-N+k+1}\\\vdots\\\xi_{-N+k+l-1}\\\theta_1\\\vdots\\\theta_k\end{array}\right), \label{(5)}
\end{equation}
where $\theta_1,\cdots,\theta_k\in\R$ in the right hand side are determined but not necessarily specified. Set\\
$$
\left(\begin{array}{c}\eta_{-N}\\\eta_{-N+1}\\\vdots\\\eta_{-N+l-1}\\\eta_{-N+l}\\\vdots\\\eta_{-N+l+k-1}\end{array}\right)
:=A_l^{-1}\left(\begin{array}{c}\xi_{-N+k}\\\xi_{-N+k+1}\\\vdots\\\xi_{-N+k+l-1}\\0\\\vdots\\0\end{array}\right).
$$
Note that $\eta_{-N},\eta_{-N+1},\cdots,\eta_{-N+l+k-1}$ are integers
and $-N+l+k-1=N$. Then, it follows that
$$
A_l\left(\begin{array}{c}x_{-N}-\eta_{-N}\\x_{-N+1}-\eta_{-N+1}\\\vdots\\x_{-N+l-1}-\eta_{-N+l-1}\\x_{-N+l}-\eta_{-N+l}\\\vdots\\x_N-\eta_N\end{array}\right)=
\left(\begin{array}{c}0\\0\\\vdots\\0\\\theta_1\\\vdots\\\theta_k\end{array}\right),
$$
From the definition of matrix $A_{l}$, hence,
$$
M_P^{(l)}\left(\begin{array}{c}x_{-N}-\eta_{-N}\\x_{-N+1}-\eta_{-N+1}\\\vdots
\\x_N-\eta_N\end{array}\right)=
\left(\begin{array}{c}0\\0\\\vdots\\0\end{array}\right).
$$
This implies that $(x_n-\eta_n:-N\le n\le N)$ is a local lift of
$(x_n:-N\le n\le N)$ in $\tilde{\Omega}_P$.

Recall that $A_l=(\tilde{L}~~L)$, thus
$$
\left(\begin{array}{c}x_{-N}-\eta_{-N}\\x_{-N+1}-\eta_{-N+1}\\\vdots\\x_{-N+l-1}-\eta_{-N+l-1}\\x_{-N+l}-\eta_{-N+l}\\\vdots\\x_N-\eta_N\end{array}\right)=
A_l^{-1}\left(\begin{array}{c}0\\0\\\vdots\\0\\\theta_1\\\vdots\\\theta_k\end{array}\right)
=(\tilde{L}~~L)\left(\begin{array}{c}0\\0\\\vdots\\0\\\theta_1\\\vdots\\\theta_k\end{array}\right)=L\left(\begin{array}{c}\theta_1\\\vdots\\\theta_k\end{array}\right)
$$
and by (\ref{(5)}),
$$
\left(\begin{array}{c}\theta_1\\\theta_2\\\vdots\\\theta_k\end{array}\right)=
R\left(\begin{array}{c}x_{-N}\\x_{-N+1}\\\vdots\\x_N\end{array}\right).
$$
Therefore,
$$
\left(\begin{array}{c}x_{-N}-\eta_{-N}\\x_{-N+1}-\eta_{-N+1}\\\vdots\\
x_N-\eta_N\end{array}\right)=LR
\left(\begin{array}{c}x_{-N}\\x_{-N+1}\\\vdots\\x_N\end{array}\right)
=\Lambda
\left(\begin{array}{c}x_{-N}\\x_{-N+1}\\\vdots\\x_N\end{array}\right),
$$ and the theorem follows.
\end{proof}

\begin{rem}{\rm
In \cite{A}, the existence of the lift is constructed directly. Here, we made a detour via
the notion of Condition $(\#)$, since it gives a concrete way to obtain it.
}\end{rem}
As a consequence of Theorem 3, we have the following surjective projection and divisibility properties:
\begin{cor}
\1 For any sufficiently large $N \in \mathbb{Z}_{>0}$, the coordinate-wise projection $\pi : \OO_P\vert{}_N \to \Omega_P\vert{}_N$
induces a surjective homomorphism of topological additive groups equipped with
the multiplicative $\mathbb{Z}$-action, where $\OO_P\vert{}_N$ is 
the set of restrictions of $(\omega_n:n\in\Z)\in\OO_P$ to 
$(\omega_n:n\in[-N,N])$. 

\2 For any stream $X=(x_n:n\in\Z)\in\Omega_P$ and any integer $m\neq 0$, there exists a stream
$Y=(y_n:n\in\Z)\in\Omega_P$ such that $mY=X$ in the coordinate-wise sense:
$$m(y_n:n\in\Z):=(my_n:n\in\Z)=(x_n:n\in\Z).$$
\end{cor}
\begin{proof}

\1 By Theorem 3, so $\Omega_P$ admits a lift. For a sufficiently large $N\in \mathbb{Z}_{>0}$, given any $(x_n)_{n \in [-N, N]} \in \Omega_P\vert{}_N$, using the lifting matrix $\Lambda$ and $\mathbf{x} = (x_{-N}, \dots, x_N)^T$, we yield a real vector $\Lambda\mathbf{x}$ in $\OO_P\vert{}_N$, whose projection under $\pi$ precisely recovers the original stream segment $(x_n)_{n \in [-N, N]}$. Since every element in $\Omega_P\vert{}_N$ has at least one pre-image (local lift) in $\OO_P\vert{}_N$, the projection map $\pi$ is surjective. Because both $\OO_P\vert{}_N$ (as a vector space over $\R$ ) and $\Omega_P\vert{}_N$ inherit the natural additive group structure, and are also compatible with the shift operator, $\pi$ is a continuous group homomorphism preserving the $\mathbb{Z}$-action.

\2 let $X:=(x_n\in\T:n\in\Z)\in\Omega_P$. By Theorem 3,
$X$ has a local lift in $\OO_P|_N$ for any large $N\in\Z$.
This means there exists a real vector sequence $(\x_n^{(N)}\in\R:-N\le n\le N)\in\OO_P|_N$ such that
$\pi(\x_n^{(N)})=x_n$ for all $n\in[-N,N]$. Since $m\in\Z\backslash\{0\}$, then
$(\x_n^{(N)}/m\in\R:-N\le n\le N)\in\OO_P|_N$.
Let $y_n^{(N)}=\pi(\x_n^{(N)}/m)\in\T$ for all $n\in[-N,N]$. Then,
$y_n^{(N)}\in\Omega_P|_N$. Since $St(\T)$ is compact under the product topology. Taking the limit of $y_n^{(N)}\in\Omega_P|_{[-N,N]}$ as $N \to \infty$, there exists a limit point,
say $Y:=(y_n)_{n\in\mathbb{Z}}\in\Omega_P$. Clearly, $Y\in\Omega_P$ and
$mY=X$, as required.
\end{proof}

\begin{lem}
For all $p, q$ and $l\in\Z_{>0}$, it holds that
$$
\det\left(\begin{array}{ccccccccc}a_1&a_0&0&\cdot&\cdot&\cdot&\cdot&\cdot&0\\
0&a_1&a_0&\cdot&\cdot&\cdot&\cdot&\cdot&0\\
\cdot&\cdot&\cdot&\cdot&\cdot&\cdot&\cdot&\cdot&\cdot\\
\cdot&\cdot&\cdot&\cdot&\cdot&\cdot&\cdot&\cdot&\cdot\\
\cdot&\cdot&\cdot&\cdot&\cdot&\cdot&\cdot&\cdot&\cdot\\
\cdot&\cdot&\cdot&\cdot&\cdot&\cdot&\cdot&\cdot&\cdot\\
\cdot&\cdot&\cdot&\cdot&\cdot&\cdot&\cdot&a_1&a_0\\
c_0&c_1&\cdot&\cdot&\cdot&\cdot&\cdot&c_{l-1}&c_l
\end{array}\right)=\det\left(\begin{array}{cc}a_1&a_0\\p&q\end{array}\right)^l,
$$
where $c_i={l\choose i}p^{l-i}q^i~(i=0,1,\cdots,l)$ and ${l\choose i}$ is binomial coefficients.
\end{lem}

By the fact that the determinant of matrix can be written as the cofactor expansion along the last row, the proof is straightforward and omitted.

\begin{ex}{\rm
Let $P(z)=3z-2$. Since $\det\left(\begin{array}{cc}3&-2\\-1&1\end{array}\right)=1$, by Lemma 2,
$\det A_4=1$ holds for
$$
A_4=\left(\begin{array}{ccccc}3&-2&0&0&0\\0&3&-2&0&0\\0&0&3&-2&0\\0&0&0&3&-2\\
1&-4&6&-4&1\end{array}\right)
$$
Here $R=(1,-4, 6, -4, 1)$ be the last row of $A_4$.
Hence, the inverse matrix $A_4^{-1}$ is an integer matrix and
$$
A_4^{-1}=\left(\begin{array}{ccccc}-5&18&-20&8&16\\-8&27&-30&12&24\\-12&40&-45&18&36\\-18&60&-68&27&54\\-27&90&-102&40&81\end{array}\right).
$$
Here $L=(16, 24, 36, 54, 81)^T$ be the last column of $A_4^{-1}$.
A lifting matrix is obtained as
$$
\Lambda=LR=\left(\begin{array}{c}16\\24\\36\\54\\81\end{array}\right)(1,-4,6,-4,1)=
\left(\begin{array}{ccccc}16&-64&96&-64&16\\24&-96&144&-96&24\\36&-144&216&-144&36\\54&-216&324&-216&54\\81&-324&486&-324&81\end{array}\right).
$$
It holds that
$\left(\cdots,\frac{14}{15},\frac{2}{5};\frac{3}{5},\frac{2}{5},\frac{1}{10},\cdots\right)\in\Omega_P$.
In fact,
$$M_P^{(4)}\left(\begin{array}{c}14/15\\2/5\\3/5\\2/5\\1/10\end{array}\right)
=\left(\begin{array}{c}2\\0\\1\\1\end{array}\right).$$
Which give a local lift
$
\Lambda\left(\begin{array}{c}14/15\\2/5\\3/5\\2/5\\1/10\end{array}\right)=\left(\begin{array}{c}344/15\\172/5\\258/5\\387/5\\1161/10\end{array}\right)$. In fact,  we have
$$
\left(\begin{array}{c}14/15\\2/5\\3/5\\2/5\\1/10\end{array}\right)\equiv\left(\begin{array}{c}344/15\\172/5\\258/5\\387/5\\1161/10\end{array}\right)~(\mbox{mod}~\Z)~\mbox{and}~
M_P^{(4)}\left(\begin{array}{c}344/15\\172/5\\258/5\\387/5\\1161/10\end{array}\right)=\left(\begin{array}{c}0\\0\\0\\0\end{array}\right).
$$

To illustrate the dynamical behavior of the systems, we can analyze their transition relations using tree structures.
For instance, if the relation $3x - 2y \equiv 0 \pmod{\mathbb{Z}}$ is viewed as a forward transformation $x \mapsto y$, it acts as a 2-valued function. Conversely, the inverse transformation $y \mapsto x$ acts as a 3-valued function.

Because of this branching behavior, forward trajectories starting from a point $x_0$ correspond to paths tracing along a 2-tree; While backward trajectories correspond to paths tracing along a 3-tree.
For example, consider the stream trajectory given by
$\left(\cdots,\frac{14}{15},\frac{2}{5};\frac{3}{5},\frac{2}{5},\frac{1}{10},\cdots\right)\in\{\Omega_P\},$
see Figures 2. The evolution of this stream moves through the tree branches as follows:

\vspace{2em}
\begin{figure}[h]
\setlength{\unitlength}{0.3mm}
\begin{picture}(150,100)
\put(198,50){$\frac{3}{5}$}

\put(210,54){\vector(2,1){30}}
\put(210,54){\vector(2,-1){30}}
\put(243,70){$\frac{9}{10}$}
\put(258,38){\vector(2,1){30}}
\put(258,38){\vector(2,-1){30}}
\put(245,36){$\frac25$}

{\color{red}\put(196,105){$x_0$}}
\put(200,76){$\vdots$}

\put(100,105){$x_{-2}$}
\put(200,76){$\vdots$}

\put(150,105){$x_{-1}$}
\put(200,76){$\vdots$}

\put(246,105){$x_{1}$}
\put(200,76){$\vdots$}

\put(290,105){$x_{2}$}
\put(200,76){$\vdots$}

\put(292,20){$\frac{1}{10}$}
\put(292,52){$\frac35$}

\put(192,54){\vector(-1,1){30}}
\put(192,54){\vector(-1,0){30}}
\put(192,54){\vector(-1,-1){30}}
\put(151,54){$\frac25$}
\put(146,84){$\frac{11}{15}$}
\put(148,18){$\frac{1}{15}$}

\put(146,56){\vector(-1,1){30}}
\put(146,56){\vector(-1,0){30}}
\put(146,56){\vector(-1,-1){30}}

\put(100,86){$\frac{14}{15}$}
\put(103,56){$\frac35$}
\put(100,26){$\frac{4}{15}$}

\put(70,105){$\cdots$}
\put(70,22){$\cdots$}
\put(70,42){$\cdots$}
\put(70,62){$\cdots$}
\put(70,82){$\cdots$}

\put(322,20){$\cdots$}
\put(322,50){$\cdots$}
\put(322,105){$\cdots$}

{\put(107,89){\circle{20}}
\put(155,57){\circle{18}}
\put(203,53){\circle{18}}
\put(248,39){\circle{18}}
\put(298,22){\circle{20}}}

\end{picture}
\caption{Forward and backward trajectories of points starting at $x_0$}
\end{figure}

\vspace{2em}

\begin{figure}[h]
\setlength{\unitlength}{0.25mm}
\begin{picture}(500,170)(-60,-20)

\multiput(0,0)(220,0){2}{
\put(0,0){\vector(1,0){180}}
\put(0,0){\vector(0,1){180}}
\put(0,150){\line(1,0){150}}
\put(150,0){\line(0,1){150}}

\put(0,0){{\color{red}\line(1,1){150}}}
\put(-10,-10){$0$}
\put(150,-13){$1$}
\put(-13,140){$1$}
\put(40,-13){$\tiny{\frac13}$}
\put(100,-13){$\frac23$}
\put(-13,75){$\frac12$}

\put(189,-3){$x$}
\put(-15,164){$y$}

\put(0,0){\line(2,3){100}}
\put(100,0){\line(2,3){50}}
\put(120,20){2}
\put(70,115){0}
\put(0,75){\line(2,3){50}}
\put(20,120){-1}
\put(50,0){\line(2,3){100}}
\put(120,95){1}}

\put(140,30){\vector(0,1){30}}
\put(140,60){\line(-1,0){80}}
\put(60,60){\vector(0,1){30}}
\put(60,90){\vector(1,0){30}}

\put(310,90){\vector(0,-1){30}}
\put(310,60){\line(-1,0){30}}
\put(280,60){\vector(0,-1){45}}
\put(280,15){\vector(-1,0){45}}

\end{picture}
\caption{The symbolic representation of trajectories for $\xi$-values.}
\end{figure}

This system can also be visualized as a trajectory on the graph of $(x,y)\in [0,1)\times[0,1)$ satisfying the relation $3x-2y=0$, where the numbers along the lines represent the $\xi$-values. In this specific case, the symbolic sequence is given by
$\xi=(\cdots,2;0,1,1,\cdots)\in\Xi_P$, see Figure 3.

Consider the inverse of the polynomial $P(z)=3z-2$. Its formal inverse in $St(\mathbb{R})$ expands as:
$$P(z)^{-1}=(3z-2)^{-1}=\frac13(z^{-1}+\frac23z^{-2}+(\frac23)^2z^{-3}+\cdots),$$
Expressed as a stream, this becomes:
$$
P(z)^{-1}=\left(\cdots,\frac13\left(\frac23\right)^3,\frac13\left(\frac23\right)^2,\frac13\left(\frac23\right),
\frac13;~0^\infty\right).
$$
By Kaneko's inverse formula, the stream can be reconstructed via convolution with the symbolic sequence $\xi\in\Xi_P$:
$$P(z)^{-1}\times\xi=(\{x_n\}\in[0,1): n\in\Z)\in\{\Omega_P\}.$$
For example, looking at the component $x_{-2} = \frac{14}{15}$, the expansion evaluates as:
$$
\frac{14}{15}=x_{-2}=2\times\frac13+0\times\frac13\left(\frac23\right)
+1\times\frac13\left(\frac23\right)^2+1\times\frac13\left(\frac23\right)^3+\cdots.
$$

The topological entropy of this dynamical system is given by $\log 3=\log 2+\log(3/2)$, which comes from two distinct contributions:
(i) the choice of branches which contributes $\log 2$ (due to the 2-valued forward mapping); (ii) the local expansion whose contribution is $\log(3/2)$ (reflecting the scaling behavior of the underlying polynomial roots).
}\end{ex}

\section{Factorization of polynomials and stream zeros}

For a preparation, we prove the following proposition which may be well known. Since we have not found references, we provide a proof here for completeness.

\begin{prop}
For a primitive polynomial $P(z)=a_kz^k+\cdots+a_1z+a_0\in\Z[z]$ with degree
$k\ge 1,~a_0a_k\ne0$ and $\Delta\in\Z_{\ge1}$,
$\Omega_P\cap\{0,\frac{1}{\Delta},\frac{2}{\Delta},\cdots,\frac{\Delta-1}{\Delta}\}^\Z$
is a finite set.
\end{prop}
It is clear that our Proposition 1 follows from the following Lemma 3.
\begin{lem}
If $X=(x_n\in\Z:n\in\Z)\in St(\Z)$ satisfies that
$P(z)\times X=(0^\infty;0^\infty)~({\rm mod}~\Delta)$, then $X~({\rm mod}~\Delta)$ is
periodic with the period bounded by a function of $k$ and $\Delta$.
\end{lem}
In this section, the period of a periodic sequence means the minimal period. Due to the theory of prime decomposition, to show this Lemma 3, it suffices to prove that $\Delta$ be a power of a prime, say, $\Delta=p^l$ for some prime $p$ and positive integer $l$. To achieve it, we need the following:

\begin{lem}
Let $p$ be a prime and $B=(b_n:n\in\Z)\in St(\Z)$ be periodic.
If $X=(x_n:n\in\Z)\in St(\Z)$ satisfies that $P(z)\times X=B~({\rm mod}~p)$.
Then, $X~({\rm mod}~p)$ is periodic with the period bounded by a function of
$k$ and the period of $B$.
\end{lem}

\noindent\textbf{Proof of Lemma 4:} Let $K$ be the period of $B$. Let $h\in\{0,1,\cdots,k\}$ be the minimum such that $a_h\ne0~({\rm mod}~p)$, its existence is guaranteed by the assumption that
$P(z)$ is primitive. Since $P(z)\times X=B~({\rm mod}~p)$,  so for any $n\in\Z$.
$$a_kx_{n-k+h}+\cdots+a_hx_n+\cdots+a_0x_{n+h}=b_{n+h}~({\rm mod}~p).$$
From the definition of $h$,
$$a_kx_{n-k+h}+\cdots+a_hx_n=b_{n+h}~({\rm mod}~p).$$
It follows that for any $n\in\Z$,
given $(x_{n-k+h},\cdots,x_{n-1})~({\rm mod}~p)$, we can see that
$x_n~({\rm mod}~p)$ is uniquely
determined by $b_{n+h}$. Here we using the assumption that $p$ is prime, $a_h\ne0~({\rm mod}~p)$, and the fact that
a linear congruence equation $a_hx=b~({\rm mod}~p)$ for any $b\in\Z$ has a unique solution in $x({\rm mod}~p)$.
Note that the number of distinct possible choices for index $n\pmod K$ is $K$, since $B$ is periodic with period $K$.

As a consequence, $x_n$ is determined as above by
$$(x_{n-k+h},\cdots,x_{n-1},~n~(\mbox{mod }K))\in\{0,1,\cdots,p-1\}^{k-h}\times\{0,1,\cdots,K-1\},$$
Hence, for any $n$ there exists $m$ with $n<m$ and $m-n\le p^{k-h}K$ such that $x_n=x_m$ which correspond to the same element in the above.
This further implies that $x_{n+1}=x_{m+1},~x_{n+2}=x_{m+2},~\cdots$.
Since $n\in\Z$ is arbitrary and $p^{k-h}K<\infty$, similarly, if we consider the case that the maximal $h\in\{0,1,\cdots,k\}$ such that $a_h\ne0~({\rm mod}~p)$, this implies that $x_{n-1}=x_{m-1},~x_{n-2}=x_{m-2},~\cdots$.
Thus we can prove that $X$ is periodic with the period at most $p^{k-h}K$, which proves Lemma 3.

Now our goal is to show this claim: If $X=(x_n:n\in\Z)\in St(\Z)$ satisfies that $P(z)\times X=(0^\infty;0^\infty)~({\rm mod}~p^l)$,
then $X~({\rm mod}~p^l)$ is periodic with the period bounded by a function of degree of $P$ and $p^l$. Once the claim has been proved, this further implies Proposition 1.
\begin{proof}[\textbf{Proof of Proposition 1}] In order to show the claim, we prove it by the induction on $l=1,2,\cdots$, provided that $\Delta=p^l$, where $p$ is a prime.
If $l=1$, then the claim is nothing but Lemma 4 with $B=(0^\infty;0^\infty)$.
Assume that the claim holds for $l=l_0\ge 1$, that is, for any stream $X=(x_n:n\in\Z),$ which satisfies $P(z)\times(x_n:n\in\Z)=(0^\infty;0^\infty)~({\rm mod}~p^{l_0})$, then $(x_n:n\in\Z)$ is periodic with the period bounded by a function of $k$ and $p^{l_0}$.

Now assume that $X=(x_n;n\in\Z)\in St(\Z)$ which satisfies
$$P(z)\times X=(0^\infty;0^\infty)~({\rm mod}~p^{l_0+1}).$$
As is well known, for any $n\in\Z$, $x_n$ can be written as $x_n=x_n'+x_n''p^{l_0}$, where $x_n'\in\{0,1,\cdots,p^{l_0}-1\}$ and $x_n''\in\Z$.
By direct computation, we obtain $$P(z)\times(x_n':n\in\Z)=(0^\infty;0^\infty)~({\rm mod}~p^{l_0}).$$
Equivalently, for any $n\in\Z$, $a_kx'_{n-k}+\cdots+a_1x'_{n-1}+a_0x'_n=0({\rm mod}~p^{l_0}),$ without loss of generality, we assume that
$b_np^{l_0}=a_kx'_{n-k}+\cdots+a_1x'_{n-1}+a_0x'_n,$
by the assumption of the induction, $(x_n':n\in\Z)({\rm mod}~p^{l_0})$ is
periodic with the period bounded by a function of $k$ and $p^{l_0}$. Clearly, $(x_n':n\in\Z)({\rm mod}~p^{l_0})=(x_n':n\in\Z)$, which follows from the definition of $x_n'$.

Now consider the stream $B:=(b_n:n\in\Z)\in St(\Z)$, then $B$ is periodic. Observe that the period of $(x_n':n\in\Z)$ is also the one of $B$, actually the period of $B$ is a factor of that of $(x'_n:n\in\Z)$.
On the other hand,
\begin{multline*}
a_kx_{n-k}+\cdots+a_1x_{n-1}+a_0x_n\\
=(a_kx''_{n-k}+\cdots+a_1x''_{n-1}+a_0x''_n+b_n)p^{l_0}=0~~({\rm mod}~p^{l_0+1}).
\end{multline*}
Hence,
$$a_kx''_{n-k}+\cdots+a_1x''_{n-1}+a_0x''_n=-b_n~~({\rm mod}~p).$$
Therefore, by Lemma 4, $(x_n'':n\in\Z)({\rm mod}~p)$ is periodic with the period bounded by a function of
$k$ and the period of $B$.
This implies that $X~({\rm mod}~p^{l_0+1})$ is periodic with a period bounded by a function
of $k$ and $p^{l_0+1}$, which is the claim for $l=l_0+1$. As the claim holds, we complete the proof of proposition.
\end{proof}

Let \(P(z), Q(z) \in \mathbb{Z}[z]\) be two primitive polynomials of degree $k$ and $h$, respectively:
\begin{equation}\label{(5')}
\begin{array}{cc}
P(z)=a_kz^k+\cdots+a_1z+a_0&(a_k\ne0,a_0\ne0)\\
Q(z)=b_hz^h+\cdots+b_1z+b_0&(b_h\ne0,b_0\ne0).
\end{array}
\end{equation}
We define the Sylvester resultant matrix $\RR$  associated with \(P(z)\) and \(Q(z)\), which is a square matrix of size \((k+h) \times (k+h)\) as follows:\\
\[
\mathcal{R} = \begin{pmatrix}
a_k    & a_{k-1} & \cdots & a_0    & 0      & \cdots \\
\vdots & \ddots  & \ddots & \ddots & \ddots & \vdots \\
0      & \cdots  & a_k    & a_{k-1} & \cdots & a_0    \\
b_h    & b_{h-1} & \cdots & b_0    & 0      & \cdots \\
\vdots & \ddots  & \ddots & \ddots & \ddots & \vdots \\
0      & \cdots  & b_h    & b_{h-1} & \cdots & b_0
\end{pmatrix}
\]

It is well known \cite{V} that $P(z),Q(z)$ are over \(\mathbb{Q}[z]\) if and only if $\RR$ is non-singular. Furthermore, if $P(z),Q(z)$ are coprime, B\'{e}zout's identity ensures the existence of polynomials  \(A(z), B(z) \in \mathbb{Z}[z]\) satisfying:
$$\deg A<h,~\deg B<k~\mbox{and}~A(z)P(z)+B(z)Q(z)=\det \RR\in\Z\setminus\{0\}.$$
Denote $\det(\RR)$ by $\operatorname{Res}(P,Q)$.

In general, for a closed $\sigma$-invariant set $\Omega\subset\T^\Z$
that is additionally invariant under the multiplicative action of $\Z$,
we define the \textbf{dimension} of $\Omega$ which measures the core degrees of freedom of the shift space, denoted by $\dim\Omega$, as the maximal integer $l$ such that
$$\{(x_0,x_1,\cdots,x_{l-1}):(x_n,n\in\Z)\in\Omega\}=\T^l.$$
Clearly, $\dim\Omega_P=k$ and $\dim\Omega_Q=h$ for the shift spaces generated by the respective polynomials $P(z),Q(z)$ in (\ref{(5')}).

For two streams $x=(x_n: n\in\Z)$ and $y=(y_n: n\in\Z)\in St(\R)$, we define their \textbf{Hadamard product} by
$$x\odot y:=(x_ny_n:n\in\Z).$$
Let $\n=(\cdots,-1;0,1,2,\cdots)\in St(\Z)$ represent the fundamental arithmetic progression stream. For any integer  \(i \in \mathbb{Z}_{\ge 1}\), we define its higher-order component via iterated Hadamard multiplication:
$$\n^{\odot i}:=\underbrace{\n\odot\cdots\odot\n}_{i\mbox{ times}}=(n^i:n\in\Z)\in St(\Z).$$
By convention, for $i=0$, $\n^{\odot i}=(1^\infty;1^\infty).$

We now show the main theorem establishing a close relationship between polynomial factorizations in the ring \(\mathbb{Z}[z]\) and decompositions of their corresponding shift-invariant subspaces.

\begin{thm}
Let $P(z),Q(z)\in\Z[z]$ be primitive polynomials as in $(\ref{(5')})$.\\
\1 $P(z)$ and $Q(z)$ are coprime over \(\mathbb{Q}[z]\) if and only if \(\operatorname{Res}(P,Q) \neq 0\). In this case, we have
$$\Omega_P\cap\Omega_Q
\subset\{\frac{0}{\Delta},\frac{1}{\Delta},\cdots,\frac{\Delta-1}{\Delta}\}^\Z$$
where $\Delta=Res(P,Q)$.
Therefore, $P(z)$ and $Q(z)$ are coprime if and only if the internection
$\Omega_P\cap\Omega_Q$ is a finite set. \\
\2 $P(z)$ is a factor of $Q(z)$ if and only if
$\Omega_P\subset\Omega_Q$. \\
\3 If $P(z),Q(z)$ are coprime, then we have
$\Omega_{P\times Q}=\Omega_P\oplus \Omega_Q$, where $\oplus$ denotes a direct sum modulo finite intersections, that is, the intersection of the summands is a finite set. \\
\4 If all roots of $P(z)=0$ are simple, then for any
\(h \in \mathbb{Z}_{\ge 1}\),, we have the semi-direct decomposition
$$
\Omega_{P^h}=\bigoplus_{i=0}^{h-1}\n^{\odot i}\odot\Omega_P.
$$
\end{thm}
\begin{proof}
\1 Let $P(z),Q(z)$ be coprime. Then, the resultant matrix $\RR$ is  invertible
and the entries of \(\Delta \mathcal{R}^{-1}\) belong to $\Z$, where
$\Delta=Res(P,Q)\in\Z\setminus\{0\}$. Since for any $j\in\Z$ and any stream
$\xi=(\xi_i:i\in\Z)\in\Omega_P\cap\Omega_Q$, we have
$\RR(\xi_{j+1},\dots,\xi_{j+k+h})^T\in\Z^{k+h}$, which implies
$$(\xi_{j+1},\dots,\xi_{j+k+h})^T\in\frac{1}{\Delta}\Z^{k+h}.$$ Therefore,
$\Omega_P\cap\Omega_Q\subset
\{\frac{0}{\Delta},\frac{1}{\Delta},\cdots,\frac{\Delta-1}{\Delta}\}^\Z$, and by Proposition 1,
$\Omega_P\cap\Omega_Q$ is a finite set.

Conversely, if $P(z),Q(z)$ are not coprime, they must share a common factor of 
degree at least 1. Hence we have
$\dim(\Omega_P \cap \Omega_Q)\ge1$, proving that the intersection $\Omega_P\cap\Omega_Q$
is infinite.

\2Assume that \(P(z)\) divides \(Q(z)\) in \(\mathbb{Z}[z]\), yielding the factorization \(Q(z) = P(z) \cdot R(z)\) for some \(R(z) \in \mathbb{Z}[z]\).
To avoid any ambiguity, the ring product "$\cdot$" (which may be omitted) is restricted to the polynomial domain \(\mathbb{Z}[z]\), whereas the convolution operator "$\times$" is reserved exclusively for the stream domain \(St(\mathbb{R})\).

Since polynomial multiplication identifies with the discrete stream convolution operator, an elements is written as \(P(z)\) when treated as a polynomials in \(\mathbb{Z}[z]\), can be also simplified to $P$ when acting as an operators within \(St(\mathbb{R})\).

Under this convention, multiplication in the polynomial ring 
\(\mathbb{Z}[z]\) corresponds directly to the discrete convolution product of their respective streams in \(St(\mathbb{R})\):
\begin{equation}
Q(z) = P(z) \cdot R(z) \in \Z[z] \quad \Longleftrightarrow \quad Q = P \times R \in St(\R).
\end{equation}
Consequently, for any $\xi\in\Omega_P$, we have
$$
Q(z)\times\xi=(P(z)R(z))\times\xi=R(z)\times(P(z)\times\xi)
=(0^\infty;0^\infty).$$ Hence, we have $\xi\in\Omega_Q$, so $\Omega_P\subset\Omega_Q$.

Conversely, if \(P(z)\) does not divide \(Q(z)\), then there exists a
nontrivial factor $R(z)$ of $P(z)$ such that $R(z)$ and $Q(z)$ are coprime.
By \1, the space $\Omega_Q\cap\Omega_R$ is finite, while $\Omega_R$ is an infinite set.
Hence, $\Omega_R\subset\Omega_Q$ does not hold.
Thus, $\Omega_P\subset\Omega_Q$ does not hold since
$\Omega_P\supset\Omega_R$, it follows directly that \(\Omega_P\subset\Omega_Q\) does not hold.

\3 By \2 and the fact that $\Omega_{P\cdot Q}$ is a additive group, we obtain
\begin{equation}\label{(6)}
\Omega_{P\cdot Q}\supset\Omega_P+\Omega_Q.
\end{equation}
Since $P(z),Q(z)$ are coprime, there exist polynomials with integer coefficients $A(z),B(z)$ with $\deg A(z)<h$ and $\deg B(z)<k$
such that
\begin{equation}P(z)A(z)+Q(z)B(z)=\Delta,
\label{(6.5)}\end{equation}
where $\Delta=Res(P,Q)\in\Z\setminus\{0\}$.
By (\ref{(6.5)}) and commutativity of convolution product, we have
\begin{equation}\label{(7)}
P\times Q\times B=P\times(\Delta-P\times A),~P\times Q\times A
=Q\times(\Delta-Q\times B).
\end{equation}

Choose an arbitrary stream $X:=(x_n:n\in\Z)\in\Omega_{P\cdot Q}$.
By Corollary 1,
there exists a pre-image stream $Y:=(y_n:n\in\Z)\in\Omega_{P\times Q}$ satisfying
$\Delta Y=X$.
By (\ref{(7)}),
$$
P\times(\Delta-P\times A)\times Y=
Q\times(\Delta-Q\times B)\times Y=(0^\infty;0^\infty).
$$
Set
$$
U=(u_n:n\in\Z):=(\Delta-P\times A)\times Y,~
V=(v_n:n\in\Z):=(\Delta-Q\times B)\times Y.
$$
By construction, $U\in\Omega_P$ and $V\in\Omega_Q$.
Furthermore, by (\ref{(6.5)}),
$$U+V=(2\Delta-P\times A-Q\times B)\times Y=\Delta Y=X. $$
Thus, we have $\Omega_{P\cdot Q}\subset\Omega_P+\Omega_Q$, hence
together with (\ref{(6)}) and \1, we complete the proof of \3.

\4 The equation
$$\Omega_{P^h}=\sum_{i=0}^{h-1}\n^{\odot i}\odot\Omega_P$$
follows from the equation
$$
\OO_{P^h}=\sum_{i=0}^{h-1}\n^{\odot i}\odot\OO_P.
$$

For any $X=(x_n:n\in\Z)\in St(\T)$, define the formal derivative $X'\in St(\R)$ as
$\sum_{n\in\Z}nx_nz^{n-1}$ as usual. That is,
$X'=\sigma(\n\odot X)$,
where $\sigma$ is the shift operator on $St(\T)$. Explicitly, this formal operation acts as:
$$(\cdots,x_{-2},x_{-1};x_0,x_1,x_{2},\cdots)'=(\cdots,-2x_{-2},-x_{-1},0; x_1, 2x_2,\cdots).$$
By identifying $X:=(x_n:n\in\Z)\in St(\T)$ with its formal power series $X(z)=\sum_{n\in\Z}x_nz^n$, then this definition of formal derivative is compatible with analytical derivative quotients.
For instance, $d(z^{-1})/dz=(z^{-1})'=-z^{-2}$, which matches the shift expression $(\cdots,1;0^\infty)'=(\cdots,-1,0;0^\infty).$

Recall that for any scalar constant $h$ and any stream $X=(x_n:n\in\Z)$,
$$hX=(0^\infty;h,0^\infty)\times(x_n:n\in\Z)=(hx_n:n\in\Z)=(\cdots,hx_{-1};hx_0,hx_1,\cdots),$$ In general, The formal derivative operator satisfies the following: $$(X+Y)'=X'+Y', (X\times Y)'=X'\times Y+X\times Y'.$$

If $X\in\Omega_P$, then $(P\times X)'=(0^\infty;0^\infty)$.
Hence, if $X\in\Omega_{P^{h-1}}$, it follows that
$$(0^\infty;0^\infty)=(P^h\times X)'=hP^{h-1}\times P'\times X+P^h\times X'=P^h\times\sigma(\n\odot X). $$
Therefore, $\sigma(\n\odot X)\in\Omega_{P^h}$, which yields $\n\odot X\in\Omega_{P^h}$.
 By iteration, we establish the nested inclusion $\n\odot\Omega_{P^{h-1}}\subset\Omega_{P^h}$. So
\begin{equation}\label{(8)}
\Omega_{P^h}\supset\sum_{i=0}^{h-1}\n^{\odot i}\odot\Omega_P~~(h=1,2,\cdots).
\end{equation}

To prove the opposite inclusion, we use the tool of local lifts from the last section.
By Corollary 1,
$$
\OO_{P^h}|_N\supset\sum_{i=0}^{h-1}\n^{\odot i}\odot\OO_P|_N
$$
holds for any large $N\in\Z_{>0}$. Since the both spaces form vector spaces over
$\R$, they coincide exactly because their dimensions are equal to $kh$.

For this purpose, we need the following lemma which is a generalization of the Vandermonde's
determinant due to R. P. Flowe and A. G. Harris \cite{FH}.

\begin{lem}{\rm\cite{FH}}
Let $l_1,\cdots,l_h$ and $L$ be positive integers such that $l_1+\cdots+l_h=L$. Let
$$
\M[x_1,\cdots,x_h:L:l_1,\cdots,l_h]:=
\left(\M[x_1,L,l_1]\M[x_2,L,l_2]\cdots\M[x_h,L,l_h]\right)
$$
be the \(L \times L\) matrix whose block columns correspond to the submatrices \(\mathcal{M}[x_i, L, l_i]\) for each \(i \in \{1, \dots, h\}\), where \(x_1, \dots, x_h \in \mathbb{C}\) and the component block matrix is given by
$$
\M(x,L,l)=\left(\begin{array}{ccccc}1&0&0&\cdots&0\\x&x&x&\cdots&x\\
x^2&2x^2&2^2x^2&\cdots&2^{l-1}x^2\\
\cdot&\cdot&\cdot&\cdots&\cdot\\\cdot&\cdot&\cdot&\cdots&\cdot\\
x^{L-1}&(L-1)x^{L-1}&(L-1)^2x^{L-1}&\cdots&(L-1)^{l-1}x^{L-1}
\end{array}\right).
$$
Then
\begin{equation}
\det\M[x_1,\cdots,x_h:L:l_1,\cdots,l_h]=C\prod_{i=1}^hx_i^{l_i(l_i-1)/2}\prod_{i<j}(x_j-x_i)^{l_il_j},
\end{equation}
where $C=\prod_{i=1}^h\prod_{j=1}^{l_i}j!\ne0$.
In particular, if \(l_i = 1\) for all \(i \in \{1, \dots, h\}\) this formulation reduces to the classical Vandermonde determinant.
\end{lem}
\noindent\textbf{Completion of the proof of \4 of Theorem 4:}~ By the assumption that all roots of $P(z)$ are simple, all roots has a multiplicity of exactly 1. Without loss of generality,
let $x_1,x_2,\cdots,x_s$ denote the distinct real roots, and let
\(y_1, \dots, y_u, \bar{y}_1, \dots, \bar{y}_u\) denote the distinct, non-real complex conjugate pairs of roots of \(P(z) = 0\), where $\y$ is the complex conjugate of $y$. It follows that $s+2u=k$ and the roots of $P(z)^h=0$ are as above with a uniform multiplicity $h$.
Hence, by Lemma 4, the family of vectors
\begin{align*}
&\{\n^{\odot i}(x_j^n:n\in\Z):i=0,1,\cdots,h-1;~j=1,2,\cdots,s\}\\
\cup&\{\n^{\odot i}(y_j^n:n\in\Z):i=0,1,\cdots,h-1;~j=1,2,\cdots,u\}\\
\cup&\{\n^{\odot i}(y_j^n:n\in\Z):i=0,1,\cdots,h-1;~j=1,2,\cdots,u\}
\end{align*}
spans a $kH$-dimensional vector space over $\C$. Let $\Re(z)$ and $\Im(z)$ denote the real and imaginary part of a complex number $z\in\C$.
Replacing the pair
$$(y_i^n:n\in\Z),~(\y_i^n:n\in\Z)~\mbox{by}~
(\Re(y_i^n):n\in\Z),~(\Im(y_i^n):n\in\Z),$$
we obrain the same real vector space over $\R$. Furthermore, this vector space is same as the vector space spanned by
$$
\OO_P,~\n\odot\OO_P,~\dots,~\n^{\odot(h-1)}\odot\OO_P.
$$
Hence, this space has a dimension $kh$ which coincides with the
dimension of the space $\OO_{P^h}$. Thus, we proved the opposite
inclusion of (\ref{(8)}) with ``$\oplus$'' instead of ``$\Sigma$''.
\end{proof}

The following Theorem is already known in
S. Akiyama, T. Kamae and H. Kaneko \cite{AKK} except for the statement of
semi-direct.

\begin{thm}
Let $P(z)\in\Z[z]$ be a primitive polynomial with the factorization into
irreducible factors
$$P(z)=Q_1(z)^{r_1}\cdots Q_m(z)^{r_m},$$
where each $Q_i(z)\in\Z[z]$ is irreducible over $\mathbb{Q}$ and $r_i\in\Z_{\geq1}$ is its multiplicity for $i = 1, \dots, m$.  Then, the stream zero set $\Omega_P$ decomposes as a semi-direct sum:
$$
\Omega_P=\bigoplus_{i=1}^m\bigoplus_{j=0}^{r_i-1}\n^{\odot j}\odot\Omega_{Q_j}.
$$

\end{thm}
Recall that a non-constant polynomial $P(z)\in\Z[z]$ is irreducible over $\Q$ if it cannot be factored into the product of two non-constant polynomials with coefficients in $\Q$. By Gauss's lemma, primitive polynomial is irreducible over $\Z$ if and only if it is irreducible over $\Q$. More generally, a primitive polynomial has the same complete factorization over $\Z$ and over $\Q$.

\begin{proof}
By \3 and \4 of Theorem 4, we have
\begin{align*}
&\Omega_P=\bigoplus_{i=1}^m\Omega_{{Q_i}^{r_i}}\\
&=\bigoplus_{i=1}^m\bigoplus_{j=0}^{r_i-1}\n^{\odot j}\odot\Omega_{Q_j}.
\end{align*}
Since the irreducible component \(Q_i(z)\) are mutually coprime, their intersections are uniformly finite.
\end{proof}

If $Res(Q,R)=\pm1$, by Theorem 5, then
$\Omega_Q\cap\Omega_R=\{(0^\infty;0^\infty)\}$ holds.
Hence, we immediately have the following corollary.
\begin{cor}
For polynomials $P(z),Q(z),R(z)$ with integer coefficients, assume that
$P(z)=Q(z)\cdot R(z)$ and  \(\operatorname{Res}(Q,R)=\pm 1\).
Then dynamical systems $(\Omega_P,\sigma)$ and $(\Omega_Q,\sigma)\otimes(\Omega_R,\sigma)$ are conjugate, where ``$\otimes$'' denotes
the direct product of dynamical systems. 
\end{cor}

\begin{rem}{\rm
As an easy example of Corollary 2, consider $Q(z)=2z-1$,  
$R(z)=3z-1$ and $P(z)=Q(z)R(z)=6z^2-5z+1$.  Then, $Res(Q,R)=1$ 
and we have $\Omega_P=\Omega_Q\oplus\Omega_R$, that is 
\begin{align*}
&\Omega_P=\Omega_Q+\Omega_R=\{(x2^n+y3^n:n\in\Z):x,y\in\R\},~
\mbox{and}\\
&\Omega_Q\cap\Omega_R=
\{(x2^n:n\in\Z):x\in\R\}\cap\{(y3^n:n\in\Z):y\in\R\}
=\{(0^\infty;0^\infty\}.
\end{align*}
Moreover, $(\Omega_P,\sigma)\cong(\Omega_Q,\sigma)
\otimes(\Omega_R,\sigma)$ holds. 
}\end{rem}

\section{Stream zeros and the strong automorphism group}

Let
\begin{equation}\label{(10)}
P(z)=a_kz^k+a_{k-1}z^{k-1}+\cdots+a_1z+a_0\in\Z[z]~(k\ge1,~a_0a_k\ne0)
\end{equation}
be a primitive polynomial and let
\begin{equation}\label{10.1}
C_P=\left(\begin{array}{ccccc}0&1&0&\cdots&0\\0&0&1&\cdots&0\\
\cdots&\cdots&\cdots&\ddots&\cdots\\0&0&\cdots&\cdots&1\\
-\frac{a_k}{a_0}&-\frac{a_{k-1}}{a_0}&\cdots&\cdots&-\frac{a_1}{a_0}\end{array}\right)
\end{equation}
be its companion matrix.
When $a_0\ne\pm1$, $C_P$ contains non-integral rational entries, causing its induced action on the torus $\T^k$ to be multi-valued. To formalize this behavior, we define the \textbf{induced action} on 
$\T^k$ of a
$k\times k$-matrix $F$ with rational entries as
$\overline{F}:=\pi\circ F\circ\pi^{-1}$,
where $\pi:\R^k\to\T^k$ is the coordinate-wise projection.
For example, if $F=(b/a)I_1$ ($I_k$ being the identity matrix of degree $k$)
with irreducible fraction $b/a$, its induced action on an element \(x \in \mathbb{T}\) is given by:
$$\overline{F}x:=\pi\left(\frac{b}{a}\times\{\{x\}+i:i\in\Z\}\right)
=\left\{\frac{\{bx\}}{a}+\frac{i}{a}:i=0,1,\cdots,a-1\right\},$$
where $\{x\}$ being the unique representative of $x\in\T$ in $[0,1)$, the braces $\{\}$ also being used to enclose set elements when there is no danger of confusion.

In general, we have
\begin{align*}
\Omega_P=&\{(x_n:n\in\Z)\in\T^\Z:a_kx_{n-k}+\cdots+a_1x_{n-1}+a_0x_n=0~
(\forall n\in\Z)\}\\
=&\left\{(x_n:n\in\Z)\in St(\T):
\left(\begin{array}{c}x_{n-k+1}\\\\\cdot\\\cdot\\x_{n-1}\\x_n\end{array}\right)\in
\overline{C_P}\left(\begin{array}{c}x_{n-k}\\\cdot\\\cdot\\x_{n-2}\\x_{n-1}\end{array}\right)~(\forall n\in\Z)\right\}.
\end{align*}

The utility of the notion of induced action which we will need 
later is in the following:
\begin{lem}
\1 The correspondence $F\mapsto\overline{F}$ defined on the set of
$k\times k$ rational matrices is one-to-one. \\
\2 Let $F$ be a $k\times k$ integral matrix satisfying $\det F=\pm1$.
Then, $F$ can be identified with $\overline{F}$ in (mod $\Z^k$).
Moreover, for any $k\times k$ rational matrix $G$, the following operator identities hold:\\
\rm(1) $\overline{FG}=F\overline{G}$\\
\rm(2) $\overline{GF}=\overline{G}F$.
\end{lem}

\begin{proof}
\1 Assume that $F\ne G$. Then,  there exists a real vector $v\in\R^k$
such that $(F-G)v$ is a nonzero irrational vector. Let $u=\pi v$,
Then, there exist two subsets
$S,S'\subset \Q^k$
such that $\overline{F}u=\pi F\pi^{-1}u=\pi(Fv+S)$ and
$\overline{G}u=\pi G\pi^{-1}u=\pi(Gv+S')$. Therefore,
$(\overline{F}-\overline{G})u=\pi(F-G)v+\pi(S-S')\ne\{\O\}$,
since $\pi(F-G)v$ is a nonzero irrational vector on $(\R/\Z)^k$
while $\pi(S-S')$ is a set of rational vectors.
Hence, $\overline{F}\ne\overline{G}$ and the correspondence
$F\mapsto\overline{F}$ is one-to-one.

\2 Since $F$ is an integer matrix, if $w$ is an integer vector, so does $Fw$. Assume $u\in\T^k$, then
for any vector $v\in\pi^{-1}u$, $Fv=F\{u\}~(\mbox{mod}~\Z^k)$.
That is, $\overline{F}u=\pi F\pi^{-1}u=F\{u\}~(\mbox{mod}~\Z^k)$.
Hence, $\overline{F}u$ is a single-point set for any $u\in(\R/\Z)^k$, therefore,
$\overline{F}$ coincide with $F$ in the sense of mod $\Z^k$.\\

For (1), note that $\pi F\pi^{-1}\pi=\pi F$. Using $\overline{F}=F(mod~\Z^k)$, hence,
$$\overline{FG}=\pi FG\pi^{-1}=\pi F\pi^{-1}\pi G\pi^{-1}=\overline{F}~\overline{G}
=F\overline{G}.$$
To show (2), take an arbitrary $u\in\T^k$, and pick a $v\in\R^k$ such that $\pi v=u$.
Since $F$ is a $k\times k$ integral matrix with $\det F=\pm1$, so $F\Z^k=\Z^k$, it holds that
$$\overline{GF}u=\pi GF\pi^{-1}u=\pi GF(v+\Z^k)=
\pi G(Fv+\Z^k)=\pi G\pi^{-1}Fu
=\overline{G}Fu.$$
Thus, $\overline{GF}=\overline{G}F$ holds.
\end{proof}

Let $\Omega_P$ be the set of stream zeros of a primitive polynomial $P(z)$ of degree \(k \ge 1\).
A mapping $\varphi:\Omega_P\to\Omega_P$ is called an
\textbf{automorphism} of $\Omega_P$ if it is a topological $\Z$-module isomorphism satisfying
$\varphi\circ\sigma=\sigma\circ\varphi$, where $\sigma$ is the shift.
Furthermore, we define $\varphi$ to be a \textbf{strong automorphism} if for any pair of streams
$X=(x_n:n\in\Z),~Y=(y_n:n\in\Z)\in\Omega_P$ satisfying $\varphi(X)=Y$, the initial blocks mutually determine each other. That is,
\begin{align*}
&(y_0,y_1,\cdots,y_{k-1})\mbox{ is uniquely determined by }(x_0,x_1,\cdots,x_{k-1}),~\mbox{and conversely,}\\
&(x_0,x_1,\cdots,x_{k-1})\mbox{ is uniquely determined by }(y_0,y_1,\cdots,y_{k-1}).
\end{align*}
In the language of symbolic dynamics, a strong automorphism is given by a local block code with a window width exactly equal to $k$ (the degree of $P(z)$ or dimension of $\Omega_P$). Because $\varphi$ commutes with the shift, this local determination holds shifted uniformly across any index $n \in \mathbb{Z}$, Explicitly, for any $n\in\Z$,
$(y_n,y_{n+1},\cdots,y_{n+k-1})$ is uniquely determined by $(x_n,x_{n+1},\cdots,x_{n+k-1})$, and vice versa.

Note that the shift $\sigma$ on $\Omega_P$ is always an automorphism, it is a strong automorphism as well only if $a_0=\pm1$ and $a_k=\pm1$.

\begin{df}{\rm
We denote by $Saut_P$ the set of $k\times k$ integer matrices $B\in \operatorname{Mat}_{k \times k}(\Z)$ that commute with the companion matrix $C_{P}$ and is unimodular, i.e.,
$$Saut_P =\left\{ B \in \operatorname{Mat}_{k \times k}(\Z) : B C_P = C_P B \text{ and } \det B = \pm 1 \right\}$$
}\end{df}

The following theorem states that each of strong automorphisms on
$\Omega_P$ can be represented by a matrix from  $Saut_P$.

\begin{thm} Assume that the polynomial $P(z)$ in $(\ref{(10)})$ is primitive. \\
\1 For any $B\in Saut_P$ with $B=(b_{i,j})_{1\le i,i\le k}$,
the entries $b_{1,k},b_{2,k},\cdots,b_{k-1,k}$ are multiples of $a_0$,
whereas  $b_{k,k}$ and $a_{0}$ are coprime. \\
\2 A mapping $\varphi:\Omega_P\to\Omega_P$ is a strong automorphism if and only if there exists
a matrix $B\in Saut_P$ such that for any stream $(x_n:n\in\Z)\in\Omega_P$ and $n\in\Z$,
\begin{equation}
\left(\begin{array}{c}y_{n}\\y_{n+1}\\\cdot\\\cdot\\y_{n+k-1}\end{array}\right)
=B\left(\begin{array}{c}x_{n}\\x_{n+1}\\\cdot\\\cdot\\x_{n+k-1}\end{array}\right), \label{(11)}
\end{equation}
 where $(y_n:n\in\Z):=\varphi(x_n:n\in\Z)$. Therefore, every strong automorphism $\varphi$
can be identified with a matrix $B\in Saut_P$.
\end{thm}

\begin{proof}
\1 Assume $a_0\ne\pm1$, otherwise there is nothing to prove. Since $BC_P=C_PB$, comparing the first rows of the matrix products yields:
$$
-\frac{a_k}{a_0}b_{1,k}=b_{2,1},~b_{1,1}-\frac{a_{k-1}}{a_0}b_{1,k}=b_{2,2},~
\cdots,~b_{1,k-1}-\frac{a_1}{a_0}b_{1,k}=b_{2,k}.
$$
Since $B$ is an integer matrix, the terms
$\frac{a_k}{a_0}b_{1,k},\frac{a_{k-1}}{a_0}b_{1,k},\dots,\frac{a_1}{a_0}b_{1,k}$ must be integers.
Since $P(z)$ is primitive, we have $GCD (a_0,a_1,\cdots,a_k)=1$,
this is possible only if $b_{1,k}$ is a multiple of $a_0$.
Repeating this row-by-row matching inductively up to the \((k-1)\)-th row establishes that
$b_{1,k},b_{2,k},\cdots,b_{k-1,k}$ are all multiples of $a_0$.
Finally, suppose for contradiction that $b_{k.k}$ and $a_0$ are not coprime. Then, $b_{1,k},b_{2,k},\cdots,b_{k-1,k}$
and $b_{k,k}$ share a common factor larger than 1. This contradicts the condition that $\det B=\pm1$.

\2 Let $\varphi:\Omega_P\to\Omega_P$ be a strong automorphism.
For any stream $X=(x_n:n\in\Z)\in\Omega_P$, let
$Y:=(y_n:n\in\Z)=\varphi(x_n:n\in\Z)\in\Omega_P$.
By definition, for any $n\in\Z$,
\begin{align*}
&(y_n,y_{n+1},\cdots,y_{n+k-1})\mbox{ is determined by }(x_n,x_{n+1},\cdots,x_{n+k-1}),
\mbox{ and}\\
&(x_n,x_{n+1},\cdots,x_{n+k-1})\mbox{ is determined by }(y_n,y_{n+1},\cdots,y_{n+k-1}).
\end{align*}
Since $\varphi$ is an automorphism and
$\Omega_P|_{\{n,n+1,\cdots,n+k-1\}}=\T^k$ for all $n\in\Z$, there exist a sequence of integer matrices $(B_{n})_{n\in\Z}$
whose inverses are also integer matrices such that
\begin{equation*}
\left(\begin{array}{c}y_{n}\\y_{n+1}\\\cdot\\\cdot\\y_{n+k-1}\end{array}\right)
=B_n\left(\begin{array}{c}x_{n}\\x_{n+1}\\\cdot\\\cdot\\x_{n+k-1}\end{array}\right)
\end{equation*}
holds for each $n\in\Z$. Since $\varphi\circ\sigma=\sigma\circ\varphi$,
we have
$$
B_1\left(\begin{array}{c}x_{1}\\x_{2}\\\cdot\\\cdot\\x_{k}\end{array}\right)
=B_1\left(\begin{array}{c}(\sigma X)_0\\(\sigma X)_1\\\cdot\\\cdot\\(\sigma X)_{k-1}\end{array}\right)=B_1B_0^{-1}\left(\begin{array}{c}(\sigma Y)_0\\(\sigma Y)_1\\\cdot\\\cdot\\(\sigma Y)_{k-1}\end{array}\right)$$
$$=B_1B_0^{-1}\left(\begin{array}{c}y_{1}\\y_{2}\\\cdot\\\cdot\\y_{k}\end{array}\right)=B_1B_0^{-1}B_1\left(\begin{array}{c}x_{1}
\\x_{2}\\\cdot\\\cdot\\x_{k}\end{array}\right). $$
Therefore we have $B_1=B_1B_0^{-1}B_1$, and hence, $B_1=B_0$ follows.
By induction, it follows that $B_n=B$ for all $n\in\Z$ for some fixed integer matrix $B$.
Because $B$ and $B^{-1}$ are both integral, so $\det B=\pm1$.

Recall that $\varphi(X)=Y$. 
Let $D=\{0,\frac{1}{a_0},\cdots,\frac{a_0-1}{a_0}\}$. 
Then we have  
$$B(x_0,\cdots,x_{k-1})^T=(y_0,\cdots,y_{k-1})^T
\mbox{ and }B(x_1,\cdots,x_k)^T=(y_1,\cdots,y_k)^T.$$

Let $\tilde{X}=(\tilde{x}_n:n\in\Z)$ and 
$\tilde{Y}=(\tilde{y}_n:n\in\Z)$ be in $\Omega_P$ such that 
$\varphi(\tilde{X})=\tilde{Y}$ and 
$(\tilde{x}_0,\tilde{x}_1,\cdots,\tilde{x}_{k-1})
=(x_0,x_1,\cdots,x_{k-1})$. Then, 
$$(\tilde{y}_0,\tilde{y}_1,\cdots,\tilde{y}_{k-1})
=(y_0,y_1,\cdots,y_{k-1})$$ holds. 
Moreover, since 
$\tilde{x}_k=x_k+d$ for some $d\in D$ and  
$\tilde{y}_k=y_k+d'$ for some $d'\in D$ and the correspondence   
$d$ to $d'$ given $x_0,x_1,\cdots,x_k$ 
is one-to-one, we have 
\begin{equation}\label{*}
B(x_1,x_2,\cdots,x_k+D)^T=(y_1,y_2,\cdots,y_k+D). 
\end{equation}

Since 
\begin{align*}
&(x_1,x_2,\cdots,x_k+D)^T
=\overline{C_P}(x_0,x_1,\cdots,x_{k-1})^T,~\mbox{and}\\
&(y_1,y_2,\cdots,y_k+D)^T=\overline{C_P}(y_0,y_1,\cdots,y_{k-1})^T, 
\end{align*}
we have 
\begin{align*}
&B\overline{C_P}(x_0,x_1,\cdots,x_{k-1})^T=B(x_1,x_2,\cdots,x_k+D)^T\\
=&(y_1,y_2,\cdots,y_k+D)^T=\overline{C_P}(y_0,y_1,\cdots,y_{k-1})^T\\
=&\overline{C_P}B(x_0,x_1,\cdots,x_{k-1})^T. 
\end{align*}
Since $(x_0,x_1,\cdots,x_{k-1})\in\T^k$ is arbitrary by taking 
$X\in\Omega_P$ arbitrary, we have $B\overline{C_P}=\overline{C_P}B$. 
Hence, by Lemma 5, we have $C_PB=BC_P$. Thus, $B\in Saut_P$. 

Conversely, let $B\in Saut_P$.
Given $X=(x_n:n\in\Z)\in\Omega_P$, 
we can define $Y=(y_n:n\in\Z)\in\Omega_P$ consistently by 
\begin{equation}\label{**}
(y_n,y_{n+1},\cdots,y_{n+k-1})^T:=
B(x_n,x_{n+1},\cdots,x_{n+k-1})^T
\end{equation}
for any $n\in\Z$. 
This is possible since we can determine $d\in D$ in the 
following equation so that $y_{n+k}=y'_{n+k}+d$ and 
(16) holds for $n+1$ instead of $n$:
\begin{multline*}
B\left(\begin{array}{c}x_{n+1}\\\cdot\\\cdot\\x_{n+k-1}\\
x_{n+k}\end{array}\right)\in B\overline{C_P}
\left(\begin{array}{c}x_n\\\cdot\\\cdot\\x_{n+k-2}\\
x_{n+k-1}\end{array}\right)
=\overline{C_P}B
\left(\begin{array}{c}x_n\\\cdot\\\cdot\\x_{n+k-2}\\
x_{n+k-1}\end{array}\right)\\
=\overline{C_P}
\left(\begin{array}{c}y_n\\\cdot\\\cdot\\y_{n+k-2}\\y_{n+k-1}\end{array}\right)=
\left(\begin{array}{c}y_{n+1}\\\cdot\\\cdot\\y_{n+k-1}
\\y'_{n+k}+D\end{array}\right),
\end{multline*}
where 
$$
y'_{n+k}\in\{(-a_ky_n-a_{k-1}y_{n+1}-\cdots-a_1y_{n+k-1})/a_0\}. 
$$
\end{proof}

Using Theorem 6, we turn to investigating the structure of $Saut_P$.
Recall that a \textbf{unit} is an algebraic integer whose inverse is also an algebraic integer. Let  

\begin{thm}
Assume that $P(z)$ is an irreducible primitive polynomial of degree \(k \ge 1\) satisfying $(\ref{(10)})$. Let $\theta_1,\cdots,\theta_k$ denote its distinct non-zero complex roots.
Then the following hold:\\
\1 A $k\times k$ matrix $B\in Saut_P$ if and only if
there exists a unit $\lambda_1$ in algebraic number field $\Q(\theta_1)$ satisfying the lattice intersection condition
\begin{equation}\label{(13.1)}
\lambda_1\in\bigcap_{i=1}^k(\Z\theta_1^{i-1}+\Z\theta_1^{i-2}+\cdots+\Z\theta_1^{i-k})
\end{equation}
such that $B$ admits the representation
\begin{equation}\label{(12.1)}
B=\left(\begin{array}{ccc}\theta_1^k&\cdots&\theta_k^k\\
\cdots&\cdots&\cdots\\\theta_1^2&\cdots&\theta_k^2\\
\theta_1&\cdots&\theta_k\end{array}\right)
\left(\begin{array}{ccc}\lambda_1&\cdots&0\\&\ddots&\\
0&\cdots&\lambda_k\end{array}\right)
\left(\begin{array}{ccc}\theta_1^k&\cdots&\theta_k^k\\
\cdots&\cdots&\cdots\\\theta_1^2&\cdots&\theta_k^2\\
\theta_1&\cdots&\theta_k\end{array}\right)^{-1},
\end{equation}
where $\lambda_2,\cdots,\lambda_k$ denote the field conjugates of $\lambda_1$ corresponding
to roots $\theta_2,\cdots,\theta_k$, respectively.
This correspondence defines a group isomorphism under matrix multiplication between \(\mathrm{Saut}\>{}_{P}\) and the group of units satisfying $(\ref{(13.1)})$, which we denote \(\Lambda _{P}\).\\
\2 It holds that 
$$\Z[a_k\theta_1]^\times\subset\Lambda_P\subset\Q(\theta_1)^\times,$$ 
where $S^\times$ denotes the set of units contained in a set $S$. 
$($It is known that if a unit is in $\Z[a_k\theta_1]^\times$, then its 
inverse is also in $\Z[a_k\theta_1]^\times$. 
Hence, $\Z[a_k\theta_1]^\times$ is a group known as the unit 
group of a integer ring $\Z[a_k\theta_1]$.$)$ \\
\3 The matrix group $Saut_P$ is isomorphic to the direct product of a finite cyclic group and $k_1+k_2-1$  copies of the infinite cyclic group  \(\mathbb{Z}\),
where $k_1, 2k_2$ are the numbers of real and non-real complex roots of $P(z)=0$, respectively.
\end{thm}

\begin{proof}
\1 Assume $B\in Saut_P$. Since
$$
C_P\left(\begin{array}{c}\theta_i^k\\\vdots\\\theta_i^2\\\theta_i\end{array}\right)
=\left(\begin{array}{c}\theta_i^{k-1}\\\vdots\\\theta_i\\1\end{array}\right)
=\theta_i^{-1}\left(\begin{array}{c}\theta_i^k\\\vdots\\\theta_i^2\\\theta_i\end{array}\right)
~(i=1,2,\cdots,k)$$
and $BC_P=C_PB$, we have
$$
C_PB\left(\begin{array}{c}\theta_i^k\\\vdots\\\theta_i^2\\\theta_i\end{array}\right)
=BC_P\left(\begin{array}{c}\theta_i^k\\\vdots\\\theta_i^2\\\theta_i\end{array}\right)
=\theta_i^{-1}B\left(\begin{array}{c}\theta_i^k\\\vdots\\\theta_i^2\\\theta_i\end{array}\right)
~(i=1,2,\cdots,k).
$$
Since the eigenspace of $C_P$ corresponding to each eigenvalue
$\theta_i^{-1}~(i=1,\cdots,k)$ is 1-dimensional, there exists a unique complex scalars
$\lambda_i\in\C$ such that
$$
B\left(\begin{array}{c}\theta_i^k\\\vdots\\\theta_i^2\\\theta_i\end{array}\right)=
\lambda_i\left(\begin{array}{c}\theta_i^k\\\vdots\\\theta_i^2\\\theta_i\end{array}\right)
~(i=1,2,\cdots,k).
$$
Thus, we have (\ref{(12.1)}). Note that $\lambda_i~(i=1,\cdots,k)$
is an algebraic integer since it is an eigenvalue of an integer
matrix $B=(b_{ij})_{i,j=1,\cdots,k}$.
Moreover,
$$
b_{i1}\theta_1^k+b_{i2}\theta_1^{k-1}+\cdots+b_{ik}\theta_1=\lambda_1\theta_1^{k-i+1}
$$
holds for any $i=1,\cdots,k$, we have
$\lambda_1\in\Z\theta_1^{i-1}+\Z\theta_1^{i-2}+\cdots+\Z\theta_1^{i-k}$,
and hence (\ref{(13.1)}).
From (\ref{(12.1)}), we also have
\begin{equation}\label{(13.2)}
\lambda_j=b_{i1}\theta_j^{i-1}+b_{i2}\theta_j^{i-2}+\cdots+b_{ik}\theta_j^{i-k}
~~(i,j=1,2,\cdots,k).
\end{equation}
Thus, $\lambda_2,\cdots,\lambda_k$ are the conjugates of $\lambda_1$
corresponding to $\theta_2,\cdots,\theta_k$.
Since $\det B=\pm1$ and (\ref{(13.1)}), we have
$\lambda_1\lambda_2\cdots\lambda_k=\pm1$, which implies that the norm
of $\lambda_1$ is $\pm1$, and hence, $\lambda_1$ is a unit satisfying (\ref{(13.1)}).

Conversely, let $\lambda_1$ be a unit satisfying (\ref{(13.1)}).
Define $B$ with the conjugates $\lambda_2,\cdots,\lambda_k$  of $\lambda_1$ by (\ref{(12.1)}).
Then $B$ is a matrix with rational entries since the right hand side
of (\ref{(12.1)}) is invariant under any substitution among
$\theta_i$'s and the correspondind $\lambda_i$'s.
Moreover, $B$ is an integer matrix since (\ref{(13.1)}),
(\ref{(13.2)}) and
the representations as (\ref{(13.2)}) is unique.
Furthermore it commutes with $C_P$ and $\det B=\pm1$, we have
$B\in Saut_P$. Thus, we have
a bijection between $Saut_P$ and the set of units in $\Q(\theta_1)$
satisfying (\ref{(13.1)}).
It is clear that this is an isomorphism with respect to the multiplications.

\2 Clearly, we have $\Lambda_P\subset\Q(\theta_1)^\times$.

Let $\alpha=a_k\theta_1$ and $\lambda_1\in\Z[\alpha]$.
Since $\alpha$ is an algebraic integer of degree $k$
and $\alpha^n\in\Z\alpha^{n-1}+\cdots+\Z\alpha^{n-k}$
for any $n\in\Z$, we have
$$\Z\alpha^n\subset\Z\alpha^{n-1}+\Z\alpha^{n-2}+\cdots+\Z\alpha^{n-k}.$$
Hence, we have
$$\lambda_1\in\Z+\Z\alpha+\cdots+\Z\alpha^{k-1}.$$
Since for any $n\in\Z$, we have
$$
a_k^n\theta_1^n=a_k^{n-1}a_k\theta_1^n
=a_k^{n-1}(-a_{k-1}\theta_1^{n-1}-\cdots-a_1\theta_1^{n-k+1}
-a_0\theta_1^{n-k}),
$$
it holds that
\begin{align*}
&\lambda_1\in\Z+\Z a_k\theta_1+\cdots+\Z a_k^{k-1}\theta_1^{k-1}\\
&\subset\Z a_k^{k-2}\theta_1^{-1}+\Z+\Z a_k\theta_1+\cdots+\Z a_k^{k-2}\theta_1^{k-2}\\
&\subset\Z a_k^{k-3}\theta_1^{-2}+\Z a_k^{k-3}\theta_1^{-1}+
\Z+\Z a_k\theta_1+\cdots+\Z a_k^{k-3}\theta_1^{k-3}\\
&\subset~\cdots\cdots\\
&\subset\Z\theta_1^{1-k}+\Z\theta_1^{2-k}+\cdots+\Z.
\end{align*}
Thus, $\lambda_1$ satisfies (\ref{(13.1)}).
Since $\Lambda_P$ is the set of units belonging to
the set of $\lambda_1$ satisfying (\ref{(13.1)}),
this implies that $\Z[\alpha]^\times\subset\Lambda_P$,
which completes the proof of \2.

\3 By the unit theorem of Dirichlet (Borevich and Shafarevich \cite{BS}),
both of $\Q(\alpha)^\times$ and $\Z[\alpha]^\times$ are direct
products of $k_1+k_2-1$ number of cyclic
groups of infinite order with a finite group.
Hence, by \2, \(\Lambda _{P}\) is sandwiched between these two groups, it inherits the identical free rank structure. The group isomorphism from part (1) lifts this property directly to \(\mathrm{Saut}\>{}_{P}\), concluding the proof.
\end{proof}

\begin{ex}{\rm
Let $P(z)=2z^2-2z-3$, which possesses the roots $(1\pm\sqrt{7})/2$.
Let $\theta=(1+\sqrt{7})/2$.
Then, $\Z[2\theta]^\times$ modulo the torsion subgroup 
is generated by $8\pm3\sqrt{7}\in\Z[2\theta]^\times$. 
Let $B\in Saut_P$ be the matrix corresponding to the eigenvalues  
$\lambda_1=8+3\sqrt{7}$ and $\lambda_1=8-3\sqrt{7}$ in (\ref{(12.1)}).
Then we have
$B=\begin{pmatrix}11&9\\6&5\end{pmatrix}$. Hence, $$Saut_P/\{\pm I_2\}=\left\{\begin{pmatrix}11&9\\6&5\end{pmatrix}^n:n\in\Z\right\},$$ where $I_2$ is the identity (unit) matrix of order 2.
}\end{ex}

\begin{ex}{\rm
Let $P(z)=2z^3-3z^2-3z+3$ which has  \(k_1 = 3\)  distinct real roots.
Let $\theta$ be one such root. Choose $\alpha=2\theta^2-\theta-5$. Since it satisfies
$\alpha^3+6\alpha^2+6\alpha-1=0$. and hence, it is a unit.
Since
$$\alpha=2\theta^2-\theta-5=2\theta-2-3\theta^{-1}=1-3\theta^{-2},$$
$\alpha$ satisfies (\ref{(13.1)}), and by (\ref{(13.2)}), $B$ in (\ref{(12.1)})
corresponding to $\lambda_1=\alpha$ is given by:
$$
B=\left(\begin{array}{ccc}1&0&-3\\2&-2&-3\\2&-1&-5\end{array}\right)
\in Saut_P.
$$
Another unit that is multiplicative independent of $\alpha$ is given by $\beta=4\theta^2-2\theta-7$
 which satisfies $\beta^3-3\beta^2-2\beta+1=0$, and corresponding matrix representation for \(\lambda_1 = \beta\) yields:
$$
C=\left(\begin{array}{ccc}5&0&-6\\4&-1&-6\\4&-2&-7\end{array}\right)
\in Saut_P.
$$
Consequently, we have
$$
Saut_P/\{\pm I_3\}=\{B^nC^m:n,m\in\Z\}.
$$
}\end{ex}

\begin{ex}{\rm
Let $\theta$ be the unique real root of the cubic polynomial $P(z)=3z^3+2z^2-3z+2$. It also has 2 complex roots.
Consider $\Z[3\theta]^\times$.

Let $\alpha=\frac32\theta^2+\frac52\theta$.
Then, $\alpha$ is a unit in $\Q(\theta)$ satisfying
$\alpha^3-2\alpha^2+4\alpha+1=0$, but it fails to satisfy (\ref{(13.1)}).
Nevertheless, $\alpha^3$ satisfies (\ref{(13.1)}), this follows from
$$
\alpha^3=3\theta^2-\theta-9=-3\theta-6-2\theta^{-1}=-4-5\theta^{-1}+2\theta^{-2}.
$$
By (\ref{(13.2)}), $B\in Saut_P$ corresponding to the eigenvalue $\lambda_1=\alpha^3$ is
$$
B=\left(\begin{array}{ccc}-4&-5&2\\-3&-6&-2\\3&-1&-9\end{array}\right)
\in Saut_P,
$$
while $B$ corresponding to $\lambda_1=\alpha$ is
$$
C=\left(\begin{array}{ccc}1/2&1/2&-1\\3/2&3/2&-1\\3/2&5/2&0\end{array}\right).
$$
Since
$$
\alpha=\frac32+\frac52\theta=\frac32\theta+\frac32-\theta^{-1}=\frac12+\frac12\theta-\theta^{-2}.
$$
Note that $C^3=B$ and $Saut_P/\{\pm I_3\}=\{B^n:n\in\Z\}$.
}\end{ex}

\begin{ex}{\rm
Let $\theta$ be a root of  the quartic polynomial $P(z)=3z^4-2z^3+z^2+z-2$ with $k_1=2$ and
$2k_2=2$.
Then, $[1,3\theta,\theta+3\theta^2,2\theta+\theta^2+3\theta^3]$ is a base of the integer ring $\Z(\theta)$ of $\Q(\theta)$.
That is,
$\Z(\theta)=\Z+3\theta\Z+(\theta+3\theta^2)\Z+(2\theta+\theta^2+3\theta^3)\Z$.
The unit group of $\Z(\theta)$, denote by $\Z(\theta)^\times$, is obtained as follows:
$$
\Z(\theta)^\times=\{\pm \alpha^m\beta^n:m,n\in \Z\}
$$
with
$$
\alpha=3\theta^3-2\theta^2-2\theta+1,~
\beta=468\theta^3+93\theta^2+237\theta+361.
$$
Both generators satisfy (\ref{(13.1)}).
In this case, $B$ and $C$ corresponding to $\lambda_1=\alpha$ and
$\lambda_1=\beta$ in (\ref{(13.1)}) respectively, are obtained as follows:
$$
B=\begin{pmatrix}
-2&3&1&-2\\
-3&0&2&0\\
0&-3&0&2\\
3&-2&-2&1\\
\end{pmatrix},~
C=\begin{pmatrix}
304&60&153&234\\
351&70&177&270\\
405&81&205&312\\
468&93&237&361\\
\end{pmatrix},
$$
Consequently,
$$
Saut_P/\{\pm I_3\}=\{B^nC^m:m,n\in \Z\}.
$$
In this case, clearly $\alpha\notin\Z[3\theta]$, hence
$\Z[3\theta]^\times\ne\Lambda_P$. 
}\end{ex}

\section{$Saut_P$ in the case $\deg P=2,3$}
In this section, we turn to investigate the structure of $Saut_P$, when the degree of the primitive polynomial $P(z)$ is restricted to $k = 2$ or $k = 3$. The first theorem is:
\begin{thm} Let $P(z)=a_2z^2+a_1z+a_0$ be a primitive quadratic polynomial, and let $D:=a_1^2-4a_0a_2$ denote its discriminant. \\
\1 Assume $D>0$, let $\theta_1,\theta_2$ be the two distinct real roots of $P(z)=0$,
for a $2\times2$ integer matrix
$B=\left(\begin{array}{cc}p&p'\\q&q'\end{array}\right)\in \operatorname{Mat}_{2 \times 2}(\Z)$, the following conditions are equivalent:\\
{\rm(i-1)}~$B\in Saut_P$,\\
{\rm(i-2)}~$B$ satisfies $(\ref{(12.1)})$ with some $\lambda_i\in\R~(i=1,2)$ such that
$\lambda_1\lambda_2=\pm1$, \\
{\rm(i-3)}~B can be written in the parametric integer form
$$B=\left(\begin{array}{cc}c_0&a_0c_1\\-a_2c_1&c_0-a_1c_1\end{array}\right)$$
for some integers $c_0,c_1\in\Z$ satisfying the Diophantine equation:
\begin{equation}\label{(14)}
(2c_0-a_1c_1)^2-Dc_1^2=\pm4,
\end{equation}
{\rm(i-4)}~$\det B=pq'-p'q=\pm1$ and
\begin{equation}
\frac{p\theta_i+p'}{q\theta_i+q'}=\theta_i~~(i=1,2).
\end{equation}
Moreover, $Saut_P/\{\pm I_2\}$
is an infinite cyclic group if and only if $D$ is a non-square integer.
In addition, if $D\in\{1,4\}$, it is a cyclic group of order $2$.
Otherwise, it is the trivial group. \\
\2 If $D=0$, then $Saut_P/\{\pm I_2\}$ is an infinite cyclic group. \\
\3 Assume $D<0$. Then $Saut_P/\{\pm I_2\}$ is the trivial group, with exactly two exceptions: If $D=-3$, then it is a cyclic group of order $3$; and if $D=-4$, then it is a cyclic group of order $2$.
\end{thm}

\begin{rem}{\rm
The equation (\ref{(14)}) is known as a generalized Pell's equation, i.e., Diophantine equation of the form
\begin{equation}
w^2-Dv^2=\pm4, \ w,v\in\Z,~\mbox{for a given integer}\,D.
\end{equation}
It seems that some structural results can be obtained by the theory of unit groups in quadratic field (or Dirichlet's unit theorem). Here we present an elementary proof utilizing the machinery of continued fractions(see \cite{T,L}, for example), however, this method breaks down when the polynomial degree increases to $k=3$, leaving only partial results for the structure of $Saut_P$ in higher dimensions.

}\end{rem}

\begin{proof}

\1 To establish the equivalence of conditions (i-1) through (i-4) under the assumption $D>0$, we verify the following implication cycle as follows:
$$\text{(i-1)}\iff \text{(i-2)}\implies\text{(i-3)}\implies\text{(i-1)}\implies \text{(i-4)}\implies \text{(i-2)}.$$
The equivalence of \rm(i-1) and \rm(i-2) follows immediately as a specific two-dimensional case of Theorem 7.
Now assume \rm(i-2) holds, then we have
\begin{equation}\label{(17)}
\left(\begin{array}{cc}p&p'\\q&q'\end{array}\right)=
\left(\begin{array}{cc}\theta_1^2&\theta_2^2\\\theta_1&\theta_2\end{array}\right)
\left(\begin{array}{cc}\lambda_1&0\\0&\lambda_2\end{array}\right)
\left(\begin{array}{cc}\theta_1^2&\theta_2^2\\\theta_1&\theta_2\end{array}\right)^{-1}
\end{equation}
with some $\lambda_i~(i=1,2)$ such that $\lambda_1\lambda_2=\pm1$.
Taking determinate operator on both sides of the equation (\ref{(17)}), we obtain
$$\det B=pq'-p'q=\lambda_1\lambda_2=\pm1.$$
Furthermore,
\begin{equation}\label{(18)}
\left(\begin{array}{cc}p&p'\\q&q'\end{array}\right)
\left(\begin{array}{cc}\theta_1^2&\theta_2^2\\\theta_1&\theta_2\end{array}\right)=
\left(\begin{array}{cc}\theta_1^2&\theta_2^2\\\theta_1&\theta_2\end{array}\right)
\left(\begin{array}{cc}\lambda_1&0\\0&\lambda_2\end{array}\right).
\end{equation}
That is
$$
\left(\begin{array}{cc}p\theta_1^2+p'\theta_1&p\theta_2^2+p'\theta_2\\q\theta_1^2+q'\theta_1&q\theta_2^2+q'\theta_2\end{array}\right)=
\left(\begin{array}{cc}\lambda_1\theta_1^2&\lambda_2\theta_2^2\\\lambda_1\theta_1&\lambda_2\theta_2\end{array}\right).
$$
So
\begin{equation}\label{(19)}
\left\{\begin{array}{c}p\theta_i^2+p'\theta_i=\lambda_i\theta_i^2\\q\theta_i^2+q'\theta_i=\lambda_i\theta_i\end{array}\right.
~\mbox{for}~i=1,2.
\end{equation}
From (\ref{(19)}), picking $c_0=p,~c'_1=p'$, we have $\lambda_i=c_0+c'_1\theta_i^{-1}, i=1,2$.

Substituting these into (\ref{(17)}), and combining with the fact that $$\theta_1\theta_2=\frac{a_0}{a_2},~\theta_1+\theta_2=-\frac{a_1}{a_2},$$ we thus have
$$
B=\left(\begin{array}{cc}p&p'\\q&q'\end{array}\right)
=\left(\begin{array}{cc}c_0&c'_1\\-\frac{a_2}{a_0}c'_1&c_0-\frac{a_1}{a_0}c'_1\end{array}\right).
$$
Since $B$ is an integer matrix, $\frac{a_2}{a_0}c'_1$ and $\frac{a_1}{a_0}c'_1$
must be integers.  By assumption that $P(z)=a_2z^2+a_1z+a_0\in\Z[z]$ is primitive, this is possible only if $c'_1$ is a multiple of $a_0$. Putting $c_1=\frac{c'_1}{a_0}\in\Z$, that is, $c_1'=a_0c_1$, thus
$$
B=\left(\begin{array}{cc}p&p'\\q&q'\end{array}\right)
=\left(\begin{array}{cc}c_0&a_0c_1\\-a_2c_1&c_0-a_1c_1\end{array}\right).
$$
Clearly, $\det B=\pm1$, which yields (\ref{(14)}). Thus, {\rm(i-2)} implies {\rm(i-3)}.

Conversely, assume that $B=\left(\begin{array}{cc}c_0&a_0c_1\\-a_2c_1&c_0-a_1c_1\end{array}\right)$ and (\ref{(14)}) holds, from the definition of disctriminat, it is easy to show that (\ref{(14)}) is equivalent to $\det B=\pm1$. By direct computation, we can verify that $BC_P=C_PB$, so $B\in Saut_p$. Thus,  {\rm(i-3)} implies {\rm(i-1)}.

Now we show that {\rm(i-1)} implies {\rm(i-4)}.
Assume $B\in Saut_p$, from the definition of $Saut_p$, $\det B=\pm1$, this gives $pq'-p'q=\pm1$.
By (\ref{(19)}), we have
$$
\frac{p\theta_i+p'}{q\theta_i+q'}=\theta_i~(i=1,2).
$$
Thus, {\rm(i-1)} implies {\rm(i-4)}.

Finally, assuming {\rm(i-4)}. Then, there exists $\lambda_1,\lambda_2$ such that
$$
\left(\begin{array}{cc}p&p'\\q&q'\end{array}\right)
\left(\begin{array}{c}\theta_i^2\\\theta_i\end{array}\right)
=\lambda_i\left(\begin{array}{c}\theta_i^2\\\theta_i\end{array}\right)~~(i=1,2).
$$
Hence, we have (\ref{(17)}), as before, $\lambda_1\lambda_2=pq'-p'q=\pm1$.
Thus, {\rm(i-4)} implies {\rm(i-2)}.

Now, we turn to  establish the group structure of $Saut_p$ by investigating the parameterized Pell-type Diophantine systems under varying discriminant domains.
First, assume $D>0$. Observe that the equation (\ref{(14)}) always has trivial solutions $c_0=\pm1,~c_1=0$,
which correspond to the identity matrices $B=\pm I_2$. If $D$ is a square integer, then non-trivial solutions can only occur when $D = 4$ yielding a cyclic group of order 2, since  the difference between two integral squares equals 4 if and only if the squares are 0 and 4.

Next, we rule out finite order elements for non-squares.
Suppose $Saut_P/\{\pm I_2\}$ contains an element of finite order, say $B$, corresponding to eigenvalues $\lambda_1=1,\lambda_2=-1,$ then according to Theorem 7,
\begin{align*}
B&=\left(\begin{array}{cc}\theta_1^2&\theta_2^2\\\theta_1&\theta_2\end{array}\right)
\left(\begin{array}{cc}1&0\\0&-1\end{array}\right)
\left(\begin{array}{cc}\theta_1^2&\theta_2^2\\\theta_1&\theta_2\end{array}\right)^{-1}\\
&=\frac{1}{\theta_1-\theta_2}
\left(\begin{array}{cc}\theta_1+\theta_2&-2\theta_1\theta_2\\2&-(\theta_1+\theta_2)\end{array}\right)
=\frac{1}{\theta_1-\theta_2}
\left(\begin{array}{cc}-\frac{a_1}{a_2}&-\frac{2a_0}{a_2}\\2&\frac{a_1}{a_2}\end{array}\right).
\end{align*}
Since $B$ is an integer matrix, which forces $\theta_1-\theta_2$ is rational. However
$$a_2^2(\theta_1-\theta_2)^2=D=a_1^2-4a_0a_2\in\Z.$$
This implies that $D$ must be a perfect square, creating a contradiction. Thus, if $D$ is a non-square integer, $Saut_P/\{\pm I_2\}$ contains only elements of infinite order except for the units.

To prove that the group is non-trivial and infinite cyclic.
let \(\eta: \operatorname{Saut}_P \to (\mathbb{R} \setminus \{0\})^2\) be the group homomorphism mapping \(B \mapsto (\lambda_1, \lambda_2)\). The image $Image(\eta)\) forms a discrete multiplicative subgroup of \((\mathbb{R} \setminus \{0\})^2\). We define the minimal positive generator:
\[
\tilde{\lambda}= \inf \left\{ \lambda_1 \in (1, \infty) : (\lambda_1, \lambda_2) \in Image(\eta) \right\}.
\]
Since $(1,-1)\notin \eta(B)$, exactly one of $(\tilde{\lambda},\tilde{\lambda}^{-1})$
or $(\tilde{\lambda},-\tilde{\lambda}^{-1})$ belong to $Image(\eta)$. This element $\tilde{\lambda}$ generates the image group $Image(\eta)$, proving that $Saut_P/\{\pm I_2\}$ is an infinite cyclic group when $D$ is a non-square integer.

To explicitly construct these infinite-order generators, we utilize the continued fraction expansion of the real root \(\theta _{1}\), which is eventually periodic. Let its continued fraction expansion is:
$$
\theta_1=[c_0;c_1,\cdots,c_{k-1},(e_0,e_1\cdots,e_{m-1})^\infty].
$$
Then,
\begin{align*}
&\frac{C}{G}=[c_0;c_1,\cdots,c_{k-1}],~\frac{C'}{G'}=[c_0;c_1,\cdots,c_{k-2}]\\
&\frac{E}{F}=[e_0;e_1,\cdots,e_{m-1}],~\frac{E'}{F'}=[e_0;e_1,\cdots,e_{m-2}],
\end{align*}
are irreducible fractions with nonnegative denominators.
With the convention, if k=0(the continued fraction expansion is purely periodic), then we put
$$\left(\begin{array}{cc}C&C'\\G&G'\end{array}\right)=
\left(\begin{array}{cc}1&0\\0&1\end{array}\right)$$
and if $k=1$ or $m=1$, we put
$$
\left(\begin{array}{cc}C&C'\\G&G'\end{array}\right)=
\left(\begin{array}{cc}c_0&1\\1&0\end{array}\right)~\mbox{and}~
\left(\begin{array}{cc}E&E'\\F&F'\end{array}\right)=
\left(\begin{array}{cc}e_0&1\\1&0\end{array}\right).
$$
We construct the $2\times 2$ matrix sequence:
$$
\left(\begin{array}{cc}p_n&p'_n\\q_n&q'_n\end{array}\right)=
\left(\begin{array}{cc}C&C'\\G&G'\end{array}\right)
\left(\begin{array}{cc}E&E'\\F&F'\end{array}\right)^n
\left(\begin{array}{cc}C&C'\\G&G'\end{array}\right)^{-1}.
$$
Then, they are mutually distinct and are the totality of elements in $Saut_P/\{\pm I_2\}$.
Hence, $Saut_P/\{\pm I_2\}$ is an infinite cyclic group.

Now assume that $D>0$ be a square. By (i-3), for an integer matrix
$B\in Saut_P$
if and only if for some $c_0,c_1\in\Z$ such that (\ref{(14)}) holds,
That is,
$$(2c_0-a_1c_1)^2-Dc_1^2=\pm4.$$
There are trivial solutions $c_0=\pm 1,~c_1=0$
corresponding to $B=\pm I$. The other solutions come from the case $c_1\ne0$.
Since the difference between two square numbers becomes to equal to 4 only if one is 0 and the other is 4, the solution of (\ref{(14)}) with $c_1\ne0$
exactly falls into the following two cases:
$$
\mbox{Case 1:}~\Big\{\begin{array}{c}c_1=\pm2\\a_1=\pm c_0\end{array} (D=1)
~\mbox{and}~
\mbox{Case 2:}~\Big\{\begin{array}{c}c_1=\pm1\\a_1=\pm 2c_0\end{array} (D=4).$$
Where the plus or minus sign $\pm$ in the above two equation systems should be coordinated in both cases. That is to say, $a_1$ is positive when $c_1$ is positive, and $a_1$ is negative when $c_1$ is negative.

In case $D=1$, we have
$B=\pm\left(\begin{array}{cc}a_1&2a_0\\-2a_2&-a_1\end{array}\right)$ which
satisfies that $B^2=I_2$ in virtue of $a_1^2-4a_0a_2=1$.
Thus, $Saut_P/\{\pm I_2\}$ is a cyclic group of order 2.

Similarly, in the case $D=4$, observe that $a_1$ is even and we have
$B=\pm\left(\begin{array}{cc}a_1/2&a_0\\-a_2&-a_1/2\end{array}\right)$ which
satisfies that $B^2=I_2$ in virtue of $a_1^2-4a_0a_2=4$.
In both cases, $Saut_P/\{\pm I_2\}$ is a cyclic group of order 2,
completing the proof of part \1.

\2
Let the discriminant $D=a_1^2-4a_2a_0=0$, by (i-3), $B\in Saut_P$ if and only if
$$B=\left(\begin{array}{cc}p&a_0p'\\-a_2p'&p-a_1p'\end{array}\right)\mbox{ with }
\det B=\left(p-\frac{a_1p'}{2}\right)^2=1.$$
Note that the leading coefficient $a_1$ is even. Put $a_1=2c$ for some integer $c$.
This reduces to $p-cp'=\pm1$, substituting these values back into $B$, we obtain
$$B=\left(\begin{array}{cc}\pm1+cp'&a_0p'\\-a_2p'&\pm1-cp'\end{array}\right)\in Saut_P.
$$
As before, the sign $\pm$ should be coordinated. Since $p'\in\Z$ is arbitrary,
we have $Saut_P/\{\pm I_2\}$ is an infinite cyclic group.

\3 Assume $D<0$,  observe that $D=a_1^2-4a_0a_2\equiv 1~\mbox{or}~0~(\rm{mod}~4)$. It follows that $D\le-3$.
Let $B\in Saut_P$, by (i-3),
$B=\left(\begin{array}{cc}c_0&a_0c_1\\-a_2c_1&c_0-a_1c_1\end{array}\right)$
for some $c_0,c_1\in\Z$ with (\ref{(14)}).
Therefore, $$(2c_0-a_1c_1)^2+|D|c_1^2=4.$$
There has a trivial solution that $c_1=0,~c_0=\pm1$, which
corresponds to $B=\pm I_2$.
Moreover, it is obvious that there are non-trivial solutions except the following two cases:

Case 1: $D=-3$ and
$\left\{\begin{array}{c}2c_0-a_1c_1=\pm1\\c_1=\pm1\end{array}\right.$
where the signs $\pm$ in the above equation system vary independently.
In this case, notice that $a_1$ must be an odd number.\\
If $c_1=1$, then $B\in Saut_P$ is one of the following:
$$
B_1=\left(\begin{array}{cc}\frac{a_1+1}{2}&a_0\\-a_2&\frac{-a_1+1}{2}\end{array}\right)
~
B_2=\left(\begin{array}{cc}\frac{a_1-1}{2}&a_0\\-a_2&\frac{-a_1-1}{2}\end{array}\right).
$$
If $c_1=-1$, then the remaining solutions for $B\in Saut_P$ are given by:
$$
B_3=\left(\begin{array}{cc}\frac{-a_1-1}{2}&-a_0\\a_2&\frac{a_1+1}{2}\end{array}\right)
~
B_4=\left(\begin{array}{cc}\frac{-a_1+1}{2}&-a_0\\a_2&\frac{a_1-1}{2}\end{array}\right).
$$
Using the assumption that $D=a_1^2-4a_0a_2=-3$, it is easy to check that
$$B_1^2=B_2,~B_1^3=-I_2~\mbox{and}~B_3=-B_1,~B_4=-B_2.$$
This implies that $Saut_P/\{\pm I_2\}$ is cyclic group of order 3. \vspace{0.5em}\\
Case 2: $D=-4$. In this case, just being the same as the case for $D=4$,
we have $Saut_P/\{\pm I_2\}$ is a cyclic group of order 2.
\end{proof}

We now give several examples to illustrate Theorem 8.

\begin{ex}{\rm
Let $P(z)=5z^2-2$ be a primitive quadratic polynomial with discriminant \(D = 40 > 0\). $\theta=\frac{\sqrt{10}}{5}$ be one of the roots of $P(z)=0$, which admits the following eventually periodic continued fraction expansion:
\[
\frac{\sqrt{10}}{5} = [0; 1, \overline{1, 1, 2}].
\]
Following the algorithm in the proof of Part \1 of Theorem 8,
\begin{align*}
&\left(\begin{array}{cc}C&C'\\G&G'\end{array}\right)
=\left(\begin{array}{cc}1&0\\1&1\end{array}\right)\\
&\frac{E}{F}:=[1;1,2]=1+\frac{1}{1+\frac12}=\frac{5}{3}~,~~\frac{E'}{F'}:=[1;1]=1+\frac{1}{1}=\frac{2}{1}.
\end{align*}
Then, we have
$$
Saut_P/\{\pm I_2\}=\left\{\left(\begin{array}{cc}1&0\\1&1\end{array}\right)
\left(\begin{array}{cc}5&2\\3&1\end{array}\right)^n
\left(\begin{array}{cc}1&0\\1&1\end{array}\right)^{-1}:n\in\Z\right\}.
$$
Setting $n=1$, we have
$$
\left(\begin{array}{cc}1&0\\1&1\end{array}\right)
\left(\begin{array}{cc}5&2\\3&1\end{array}\right)
\left(\begin{array}{cc}1&0\\1&1\end{array}\right)^{-1}
=\left(\begin{array}{cc}3&2\\5&3\end{array}\right).
$$
Evaluating the action of $B$ on a valid stream
$$X=(\cdots;\frac37,\frac57,\frac47,\frac27,\frac37,\frac57,\cdots)\in\Omega_P,$$
induces a coding mapping, denoted by $\varphi_B$, that outputs the stream:
$$Y=\varphi_B(X)=(\cdots;\frac57,\frac27,\frac27,\frac57,\frac57,\frac27,\cdots)\in\Omega_P.$$
}\end{ex}

\begin{ex}{\rm
Let $P(z)=z^2-2z+1$ with its discriminant $D=0$.
Then by Theorem 8, $Saut_P=\{\pm B^n:n\in\Z\}$ with
$B=\left(\begin{array}{cc}0&1\\-1&2\end{array}\right)$.
Mapping on a stream $(x_n:n\in\Z)\in\Omega_P$, we get
$(y_n:n\in\Z)\in\Omega_P$ with
$$
\left(\begin{array}{c}y_n\\y_{n+1}\end{array}\right)=
\left(\begin{array}{cc}0&1\\-1&2\end{array}\right)
\left(\begin{array}{c}x_n\\x_{n+1}\end{array}\right)=
\left(\begin{array}{c}x_{n+1}\\-x_n+2x_{n+1}\end{array}\right).
$$
We find that $y_{n}=x_{n+1}, y_{n+1}=-x_n+2x_{n+1}$, further, $$y_{n+2}=-y_n+2y_{n+1}=-2x_n+3x_{n+1},$$ which is compatible with the matrix product:
$$
\left(\begin{array}{c}y_{n+1}\\y_{n+2}\end{array}\right)=
\left(\begin{array}{cc}0&1\\-1&2\end{array}\right)
\left(\begin{array}{c}x_{n+1}\\-x_n+2x_{n+1}\end{array}\right)=
\left(\begin{array}{c}-x_n+2x_{n+1}\\-2x_n+3x_{n+1}\end{array}\right).
$$
}\end{ex}

\begin{ex}{\rm
Let $P(z)=z^2+1\in\Z[z]$ with its discriminant $D=-4$.
Then it is easy to show that
$Saut_P/\{\pm I_2\}=\left\{I_2,\left(\begin{array}{cc}0&1\\-1&0\end{array}\right)\right\}$.
Since $\left(\begin{array}{cc}0&1\\-1&0\end{array}\right)^2=I_2$,
$Saut_P/\{\pm I_2\}$ is a cyclic group of order 2.
}\end{ex}

We now discuss the case of degree $k=3$. It's much harder than the case of degree 2 and we have a result only for a special case.

\begin{thm}
\1 Let $P(z) = a'_3 z^3 + a'_2 z^2 + a'_1 z + a'_0 \in \Z[z], a'_0 a'_3 \neq 0$ be a primitive cubic polynomial. Assume that $P(z)$ is simple with distinct roots $\theta_1, \theta_2, \theta_3$. Define $a_3 = \frac{a'_3}{a'_0}$, $a_2 = \frac{a'_2}{a'_0}$, and $a_1 = \frac{a'_1}{a'_0}$.

Then $B\in Saut_P$ if and only if
$$B=\left(\begin{array}{ccc}c_0&c_1&c_2\\-a_3c_2&c_0-a_2c_2&c_1-a_1c_2\\-a_3c_1+a_1a_3c_2&-a_2c_1+(a_1a_2-a_3)c_2&c_0-a_1c_1+(a_1^2-a_2)c_2\end{array}\right)$$
for some integers $c_0,c_1,c_2\in\Z$ such that
\begin{equation*}
\begin{split}
 ~~~\det B&=c_0^3-a_3c_1^3+a_3^2c_2^3-a_1c_0^2c_1+(a_1^2-4a_2)c_0^2c_2+a_2c_0c_1^2\\&+(a_2^2-2a_1a_3)c_0c_2^2+a_1a_3c_1^2c_2-a_2a_3c_1c_2^2+(3a_3-a_1a_2)c_0c_1c_2\\&=\pm1. \end{split}
\end{equation*}
\2 Let $P(z)=-az^3+1\in\Z[z]$ with $a\neq 0$. Then, $B\in Saut_P$ if and only if
$B=\left(\begin{array}{ccc}c_0&c_1&c_2\\ac_2&c_0&c_1\\ac_1&ac_2&c_0\end{array}\right)$
for some $c_0,c_1,c_2\in\Z$ such that
\begin{equation}
\det B=c_0^3+ac_1^3+a^2c_2^3-3ac_0c_1c_2=\pm1.
\end{equation}
\3
$P(z)=-z^3+1\in\Z[z]$. Then, $Saut_P/\{\pm I_3\}$ is a finite cyclic group of order $3$ generated by
$\left(\begin{array}{ccc}0&1&0\\0&0&1\\1&0&0\end{array}\right)$.
\end{thm}

\begin{proof}
\1
By hypothesis, the inverse roots $\theta_i^{-1}(i=1,2,3)$ is the roots of the reciprocal polynomial $h(z):=z^3+a_1z^2+a_2z+a_3$. Put
$\tau_1:=\frac{1}{\theta_1}+\frac{1}{\theta_2}+\frac{1}{\theta_3}$,
$\tau_2:=\frac{1}{\theta_1\theta_2}+\frac{1}{\theta_1\theta_3}+\frac{1}{\theta_2\theta_3}$,
$\tau_3:=\frac{1}{\theta_1\theta_2\theta_3}$.

Applying theorem 7 with $k=3$,
\begin{equation*}
B=\left(\begin{array}{ccc}\theta_1^3&\theta_2^3&\theta_3^3\\\theta_1^2&\theta_2^2&\theta_3^2\\
\theta_1&\theta_2&\theta_3\end{array}\right)
\left(\begin{array}{ccc}\lambda_1&0&0\\0&\lambda_2&0\\0&0&\lambda_3\end{array}\right)
\left(\begin{array}{ccc}\theta_1^3&\theta_2^3&\theta_3^3\\\theta_1^2&\theta_2^2&\theta_3^2\\
\theta_1&\theta_2&\theta_3\end{array}\right)^{-1}
\end{equation*}
where $\lambda_i=c_0+c_1\theta_i^{-1}+c_{2}\theta_i^{-2}~(i=1,2,3).$
A careful calculation gives
\begin{equation*}
B=
\left(\begin{array}{ccc}c_1&c_2&c_3\\c_2\tau_3&c_0-c_2\tau_2&c_1+c_2\tau_1\\
c_1\tau_3+c_2\tau_1\tau_3&-c_1\tau_2-c_2\tau_1\tau_2+c_2\tau_3&c_0+c_1\tau_1+c_2(\tau_1^2-\tau_2)\end{array}\right).
\end{equation*}
 From Vieta's formulas, we have
$\tau_1=-a_1,\tau_2=a_2,\tau_3=-a_3.$
The second statement follows just by calculating $\det B$, so \1 follows.

\2
From $BC_P=C_PB$, so the expression of matrix $B$ in \2 follows. Or applying \1 with $a_1=a_2=0, a_3=-a.$

\3
By \2, $B\in Saut_P$ if and only if
\begin{align*}
&\det B=c_0^3+c_1^3+c_2^2-3c_0c_1c_2\\
&=(c_0+c_1+c_2)(c_0^2+c_1^2+c_2^2-c_0c_1-c_1c_2-c_2c_1)\\
&=\frac12(c_0+c_1+c_2)((c_0-c_1)^2+(c_1-c_2)^2+(c_2-c_0)^2).
\end{align*}
This forces $c_0+c_1+c_2=\pm1$. Hence,
$(c_0-c_1)^2+(c_1-c_2)^2+(c_2-c_0)^2=2$.
This is possible only if two of the terms in the left hand side is 1 and the other is 0.
Without loss of generality, assume that $c_0=c_1$ and $c_2=c_0\pm1$.
Then, we have $c_0+c_1+c_2=3c_0\pm1$. Since $c_0+c_1+c_2=\pm1$, we must have
$c_0=0$, and hence, $(c_0,c_1,c_2)=(0,0,\pm1)$. Thus, $\lambda_1\lambda_2\lambda_3=\pm1$
is equivalent to $(c_0,c_1,c_2)\in\{(0,0,\pm1),(0,\pm1,0),(\pm1,0,0)\}$, and hence,
$(\lambda_1,\lambda_2,\lambda_3)\in\{(\pm1,\pm1,\pm1),(\pm\omega,\pm\omega,\pm\omega),
(\pm\omega^2,\pm\omega^2,\pm\omega^2)\}$.
Thus, $Saut_P/\{\pm I_2\}$ is a cyclic group of order 3.
\end{proof}

\begin{ex}{\rm
Let $P(z)=-z^3+1\in\Z[z]$ and $\omega=\frac{-1+\sqrt{3}}{2}$.
Then, by Theorem 9, $B=\left(\begin{array}{ccc}0&1&0\\0&0&1\\1&0&0\end{array}\right)$.
We verify that
$$B^2=\left(\begin{array}{ccc}0&0&1\\1&0&0\\0&1&0\end{array}\right)\neq I_3,~B^3=I_3 .$$
Therefore, $Saut_P/\{\pm I_3\}$ is a cyclic group of order $3$, which generated by matrix
$B$. Let two streams $X=(x_n:n\in\Z), Y=(y_n:n\in\Z)$ belong to $\Omega_P$ such that $\varphi_B(X)=Y$, where $\varphi_B$ is the coding mapping induced by $B$, then $\varphi_B$ maps the  initial segment of $X$ to:
$$
\left(\begin{array}{c}y_0\\y_1\\y_2\end{array}\right)=
\left(\begin{array}{ccc}0&1&0\\0&0&1\\1&0&0\end{array}\right)
\left(\begin{array}{c}x_0\\x_1\\x_2\end{array}\right)=
\left(\begin{array}{c}x_1\\x_2\\x_0\end{array}\right).
$$
Since $x_3=x_0,~y_3=y_0$, we have
$$
\left(\begin{array}{c}y_1\\y_2\\y_3\end{array}\right)=
\left(\begin{array}{ccc}0&1&0\\0&0&1\\1&0&0\end{array}\right)
\left(\begin{array}{c}x_1\\x_2\\x_3\end{array}\right)=
\left(\begin{array}{c}x_2\\x_3\\x_1\end{array}\right)=
\left(\begin{array}{c}x_2\\x_0\\x_1\end{array}\right),
$$
which matches the expected shift-equivariant behavior.
}\end{ex}

\section{Streams over Galois fields}

In the final section, we investigate streams over the Galois field  
$\text{GF}(q)$ ($q$ is a prime power) instead of $\T=\R/\Z$. 
Let $P(z)\in \text{GF}(q)[z]$ be an irreducible polynomial over $\text{GF}(q)$ of degree $k\ge1$, and let $\Omega_P^q$ denote the set of stream zeros in $ST(\text{GF}(q))$ under the action of $P(z)$. Note that any element $(x_n:n\in\Z)\in\Omega_P^q$ is periodic with a period at most $q^k$, as it is uniquely determined by its initial segment $(x_0,x_1,\dots,x_{k-1})\in \text{GF}(q)^k$ where $\#\text{GF}(q)=q$.

\begin{thm}
$Saut_P$ is a cyclic group of order $q^k-1$.
\end{thm}

\begin{proof}
Let $P(z)=a_0+a_1z+\cdots+a_kz^k$, where 
$a_i\in \text{GF}(q)~(i=0,1,\cdots,k)$, with $a_0 \neq 0$ and $a_k \neq 0$. Let $\theta_1,\theta_2,\cdots,\theta_k$ be the roots of $P(z)=0$.
Since $P(z)$ is irreducible, these roots are distinct from each other.
Let $\eta$ be a primitive root of the extension field $\text{GF}(q^k)=\text{GF}(q)(\theta_1,\cdots,\theta_k)$,
so that $\eta^i=1$ if and only if $i\equiv 0~({\rm mod}~q^k-1)$.
Then, $\eta$ can be expressed as a $\text{GF}(q)$-linear combination of powers of $\theta_1$:
$$\eta=c_1\theta_1+\cdots+c_k\theta_1^k$$ where
$c_1,\cdots,c_k\in \text{GF}(q)$.

For each \(i=1,2,\dots,k\), we define
$$\lambda_i=c_1\theta_i+\cdots+c_k\theta_i^k~(i=1,\cdots,k).$$
Let the matrix $B$ be defined as in equation (\ref{(12.1)}).
Since the entries are symmetric functions of 
$\theta_1, \cdots,\theta_k$. Hence, $B$ is a matrix with entries entirely in $\text{GF}(q)$. 
Furthermore, $B$ commutes with the companion matrix $C_{P}$, which implies that $B\in Saut_P$.

Moreover, because the eigenvalue $\lambda_1=\eta$ has multiplicative order $q^k-1$, the matrix powers satisfy $B^i=I_k$ if and only if $i\equiv 0~(mod~q^k-1)$. Given that the cardinality of the set of nonzero matrices $B$ of the form as (\ref{(12.1)}) is at most $q^k-1$, this implies that $Saut_P$ is a cyclic group of order $q^k-1$. 
\end{proof}

\begin{ex}{\rm
Consider the polynomial $P(z)=z^2+1\in\text{GF}(3)[z]$ of degree $2$. The Galois field  $\text{GF}(3)$ is the finite field $\{0,1,2\}$ under arithmetic modulo $3$. By direct computation, $P(0)=1,P(1)=2,P(2)=1$, since $P(z)$ has no roots in $\text{GF}(3)$, it follows that $P(z)$ is irreducible.

Let $\pm\theta\in \text{GF}(9)$ be the roots of $z^2+1=0$. Then, 
$\theta+1$ becomes a primitive root of $\text{GF}(9)$. 
The matrix $B$ defined below generates $Saut_P$ 
$$
B=\left(\begin{array}{cc}2&2\\\theta&-\theta\end{array}\right)
\left(\begin{array}{cc}\theta+1&0\\0&-\theta+1\end{array}\right)
\left(\begin{array}{cc}2&2\\\theta&-\theta\end{array}\right)^{-1}
=\left(\begin{array}{cc}2&1\\2&2\end{array}\right),~\mbox{and}
$$  
$$
B^n=\left(\begin{array}{cc}2&2\\\theta&-\theta\end{array}\right)
\left(\begin{array}{cc}(\theta+1)^n&0\\0&(-\theta+1)^n\end{array}\right)
\left(\begin{array}{cc}2&2\\\theta&-\theta\end{array}\right)^{-1}, 
where $$
$(\theta+1)^n~(n=1,\cdots,3^2-1)$ take all the nonzero values 
in $\text{GF}(3^2)$. 

\vspace{7em}
\begin{figure}[h]
\setlength{\unitlength}{0.2mm}
\begin{picture}(0,0)(-150,0)
\multiput(0,0)(200,0){2}{
\put(0,0){\vector(1,0){120}}
\put(0,0){\vector(0,1){120}}
\put(0,50){\line(1,0){100}}
\put(0,100){\line(1,0){100}}
\put(50,0){\line(0,1){100}}
\put(100,0){\line(0,1){100}}

\put(-5,-25){0}
\put(45,-25){1}
\put(95,-25){2}
\put(115,-16){$x_n$}

\put(-25,-8){0}
\put(-25,42){1}
\put(-25,92){2}}
\put(-43,115){$x_{n+1}$}

\put(50,0){\thicklines\vector(-1,2){50}}
\put(0,100){\thicklines\vector(1,-1){100}}
\put(100,0){\thicklines\vector(-2,1){100}}
\put(0,50){\thicklines\vector(1,-1){50}}
\put(158,115){$x_{n+1}$}

\put(200,0){
\put(50,50){\thicklines\vector(0,1){50}}
\put(50,100){\thicklines\vector(1,0){50}}
\put(100,100){\thicklines\vector(0,-1){50}}
\put(100,50){\thicklines\vector(-1,0){50}}}
\end{picture}
\vspace{1em}
\caption{$\sigma$-orbits of $\Omega_P^3$
other than $(0^\infty)$}
\end{figure}
\vspace{1em}

\begin{figure}[h]
\setlength{\unitlength}{0.3mm}
\begin{picture}(0,0)(-150,100)
\put(0,0){\vector(1,0){120}}
\put(0,0){\vector(0,1){120}}
\put(0,50){\line(1,0){0}}
\put(0,100){\line(1,0){0}}
\put(50,0){\line(0,1){0}}
\put(100,0){\line(0,1){0}}

\put(-5,-20){0}
\put(45,-20){1}
\put(95,-20){2}
\put(115,-13){$x_n$}

\put(-20,-6){0}
\put(-20,44){1}
\put(-20,94){2}
\put(-30,115){$x_{n+1}$}

\put(50,0){\thicklines\vector(1,2){50}}
\put(100,100){\thicklines\vector(-1,0){100}}
\put(0,100){\thicklines\vector(2,-1){100}}
\put(100,50){\thicklines\vector(0,-1){45}}
\put(100,5){\thicklines\vector(-1,1){45}}
\put(55,50){\thicklines\vector(-1,0){55}}
\put(0,50){\thicklines\vector(1,1){45}}
\put(45,95){\line(1,0){5}}
\put(50,95){\thicklines\vector(0,-1){95}}

\end{picture}
\vspace{10em}
\caption{$B$-orbit of $\Omega_P^3$
other than $(0,0)^\infty$}
\end{figure}

Any element $(x_n:n\in\Z)\in\Omega_P^3$ is determined by 
$(x_0,x_1)\in\{0,1,2\}^2$,  so that we represent it by a point 
in $\{0,1,2\}^2$. Figure 5 illustrate 
all the $\sigma$-orbits of $\Omega_P^3$ other than $(0,0)^\infty$. 
The left hand side implies the $\sigma$-orbits
$((1,0)(0,2)(2,0)(0,1))^\infty$, while 
the right hand side does $((1,1)(1,2)(2,2)(2,1))^\infty$. 

Figure 6 represents the orbit of nonzero elements in $\Omega_P^3$ by 
the group of strong automorphisms which is generated by the element 
corresponding to $B$. That is, 
$((1,0)(2,2)(0,2)(2,1)(2,0)(1,1)(0,1)(1,2))^\infty$. 

}\end{ex}

\noindent{\bf Acknowledgment:}
The authors thank the referee for thoroughly reviewing the manuscript and providing valuable suggestions. The authors are also grateful to Prof. Hajime Kaneko from Tsukuba University in Japan for his helpful comments; in fact, Theorem 5 is essentially due to him. The second author was supported by the TianYuan Visiting Scholar Program of the National Natural Science Foundation of China (NSFC) (Grant No:12426661).
He would also like to express his gratitude to the Institute of Advanced Mathematics
of Osaka Metropolitan University and the South China University of Technology
for providing supportive and stimulating academic environments for this work.


\begin{thebibliography}{9}

\bibitem{A} S. Akiyama, Number system based on non-integral algebraic numbers(preprint).

\bibitem{AKK} S. Akiyama, T. Kamae and H. Kaneko, Exponential diophantine approximation and symbolic dynamics, Mathematische Zeitschrift 311 (2025).

\bibitem{BR} M. Baake and J. A. G. Roberts, Symmetries and reversing symmetries of toral automorphisms, Nonlinearlity 14 (2001) R1-R24.

\bibitem{BS} Z. I. Borevich and I. R. Shafarevich, Number Theory, Academic Press, New York (1966).

\bibitem{BLR} M. Boyle, D. Lind and D. Rudolph, The automorphism group of a shift of finite type, Trans. Amer. Math. Soc. 306 (1988), 71-114.

\bibitem{C} T. Catalan, A link between topological entropy and Lyapunov exponents,  Ergodic Theory Dyn. Syst.39(3) (2019), 620-637.

\bibitem{CK} V. Cyr and B. Kra. The automorphism group of a shift of linear growth: beyond transitivity. Forum Math. Sigma 3 (2015): 27.

\bibitem{FH} R. P. Flowe and A. G. Harris, A Note on Generalized Vandermonde determinants, SIAM J. Matrix Anal. Appl. 14 (1993), 1146-1151.

\bibitem{H} M. Hochman, On the automorphism groups of multidimensional shifts of finite type, Ergodic Theory Dynam. Systems 30(3) (2010), 809-840.
\bibitem{L} H.W.Jr. Lenstra, Solving the Pell Equation, Notices of the American Mathematical Society, 49 (2), 2002.

\bibitem{N} M. Newman, Integral matrices, Academic Press, New York, 1972, Pure and Applied Mathematics, Vol. 45.

\bibitem{LM} D. Lind and B. Marcus. An introduction to symbolic dynamics and coding. Cambridge Mathematical Library. Cambridge University Press, Cambridge, 2021. Second edition.

\bibitem{KK} B. Kitchens and K. Schmidt. Automorphisms of compact groups. Ergod. Th. \& Dynam. Sys. 9 (1989), 691-735.

\bibitem{S} K. Schmidt, Dynamical Systems of Algebraic Origin, Progress in Mathematics, Vol. 128, Birkhauser, Basel, 1995.

\bibitem{KKS} A. Katok, S. Katok and K. Schmdit, Rigidity of measurable structure for algebraic actions of higher-rank abelian groups . Comment. Math. Helv. 77 (2002), 718-745.

\bibitem{R} J. J. M. M. Rutten, Elements of Stream Calculus (An Extensive Exercise in Coinduction), Electronic notes in Theoretical Computer Science 45 Elsevier 2001.

\bibitem{T} T. Takagi, Lectures on Elementary Number Theory, Kyoritsu Shuppan, Tokyo, 1971 (Japanese).

\bibitem{V} B. Van der Waerden, Algebra. Vol 1, Translated by Fred Blum and John R. Schulenberger, Frederick Ungar Publishing Co., New York, 1970 xiv+265.

\bibitem{Z} X. Z. Zhan, Completion of a partial integral matrix to a unimodular matrix, Linear Algebra and its Applications 414 (2006) 373-377.

\end{thebibliography}
\end{document}